\documentclass[12pt]{amsart}
\usepackage{fullpage}
\usepackage[english]{babel}
\usepackage{fancyhdr,picins,hyperref,float}
\usepackage{amsthm}
\usepackage{mathtools}
\usepackage{amsfonts}
\usepackage{amssymb}
\usepackage{enumitem}
\usepackage{graphicx}
\usepackage{caption}
\usepackage{xypic}

\newtheorem{prop}{Proposition}
\newtheorem{lem}{Lemma}

\theoremstyle{definition}
\newtheorem{df}{Definition}

\newtheorem{rmk}{Remark}
\newtheorem*{mainthmA}{Main Theorem A}
\newtheorem*{mainthmB}{Main Theorem B}

\newcommand{\mE}{\mathcal{E}}
\newcommand{\mD}{\mathcal{D}}

\newcommand{\ZZ}{\mathbb{Z}}
\newcommand{\mF}{\mathcal{F}}
\newcommand{\f}{\mathfrak{f}}
\newcommand{\mG}{\mathcal{G}}
\newcommand{\mA}{\mathcal{A}}
\newcommand{\mL}{\mathcal{L}}
\newcommand{\mV}{\mathcal{V}}
\newcommand{\G}{\Gamma}
\newcommand{\LDP}{Lonely Direction Property}
\newcommand{\vphi}{\varphi}
\newcommand{\veps}{\varepsilon}
\newcommand{\mT}{\mathcal{T}}
\newcommand{\teich}{Teichm\"{u}ller }

\newcommand{\from}{\colon}
\newcommand{\out}{\textup{Out}(F_r)}

\newcommand{\os}{\rm{CV}}
\newcommand{\cv}{\os}

\begin{document}

\title{Taking the high-edge route of rank-$3$ Outer space}
\author{Damara Gagnier, Catherine Pfaff}

\thanks{The authors acknowledge support from NSERC (the Natural Sciences and Engineering Research Council of Canada) and the Queen's University Department of Mathematics and Statistics.}

\maketitle

\begin{abstract}
Principal fully irreducible outer automorphisms were introduced in \cite{stablestrata} to emulate
principal pseudo-Anosov surface homeomorphisms,
i.e. those whose attracting and repelling invariant foliations have only 3-pronged singularities.
It is proved in \cite{stablestrata} that these outer automorphisms semi-mimic principal pseudo-Anosov surface homeomorphisms
in important ways. We prove here which rank-3 graphs carry train track representatives of
principal fully irreducible outer automorphisms.
As a corollary, one obtains which simplices of rank-3 Culler-Vogtmann Outer space principal axes pass through.
 \end{abstract}

\section{Introduction}
Let $\out$ denote the outer automorphism group of a fixed free group
of rank $r\geq 3$. This group's study has fundamental
contributions dating back to Nielsen, starting in 1915, and Whitehead
in the 1930's to 1950's.  Recent substantial progress has utilized that, just as the mapping class group of a surface encodes all
(equivalence classes of) homeomorphisms of the surface, $\out$ encodes
all (equivalence classes of) homotopy equivalences of graphs of a
fixed rank.  In particular, this provides a powerful analogue of
Teichm\"{u}ller space, namely Outer space (denoted $CV_r$), the deformation space of marked metric graphs. This correspondence further gives rise to related analogues of the train track theory.

Pseudo-Anosov surface homeomorphisms are the generic (in a random walk sense, see \cite{r08,KM98,cm10,m11,Sisto,horoboundary,MaherSisto}) and dynamically minimal elements of a mapping class group. The $\out$-analogues are fully irreducible outer automorphisms. Various methods for constructing fully irreducible outer automorphisms with an assortment of properties can be found, for example, in \cite{cp10}, \cite{IWGII}, \cite{IWGIII}, and \cite{c15}. We focus on ``principal'' fully irreducible outer automorphisms, as introduced in \cite{stablestrata} to emulate principal pseudo-Anosov surface homeomorphisms. Principal pseudo-Anosov homeomorphisms were proved generic (in a random-walk sense) in \cite{gm16}. Principal structures in the pseudo-Anosov and \teich space situation play an important role. For each genus, there is a partial ordering on the set of singularity structures, and the 3-pronged (i.e. principal) structure is the unique maximal element. This partial ordering has dynamical significance in that, in the total geodesic flow of the moduli space, the closure of a stratum associated to one singularity structure contains the closure of a lower stratum associated to a second singularity structure if and only if the first structure is greater than the second in the partial ordering. Principal pseudo-Anosovs homeomorphisms play a fundamental role in the work of Masur, for example, in \cite{m82} The existence of at least one principal fully irreducible outer automorphism in each rank was proved in \cite{stablestrata}.

\begin{mainthmA}
The graphs carrying train track representatives of principal fully irreducible outer automorphisms in $Out(F_3)$ are precisely:
~\\
\vspace{-5mm}
\begin{figure}[H]
\centering
\includegraphics[width=4in]{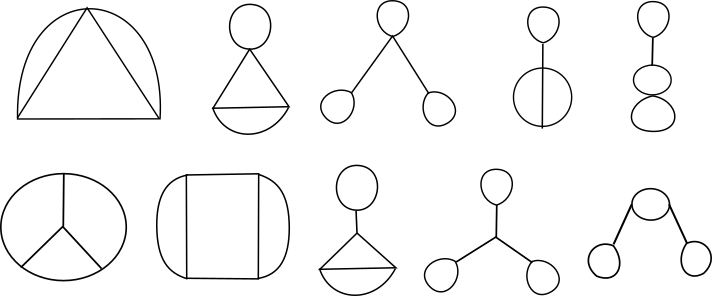}
\label{fig:21Graphs}
\end{figure}
That is, each of the graphs of $CV_3$ where either all vertices are valence-3 or all vertices are valence-3 except that one is valence-4 carries a train track representative a of principal fully irreducible outer automorphism in $Out(F_3)$.\footnote{The complete automaton can be found written out at {\url{https://mast.queensu.ca/~cpfaff/Automaton}}.}
\end{mainthmA}

A fully irreducible outer automorphism in $\out$ (more specifically Stallings fold decompositions of train track representatives of the outer automorphism, please see Definition \ref{d:folds}) determines geodesics (called axes) in $CV_r$. In \cite{stablestrata}, it is shown that the axes of principal fully irreducible outer automorphisms share a certain ``stability'' property with principal  pseudo-Anosov axes in Teichm\"{u}ller space. This stability property is then used in \cite{randomout} and \cite{hittingmeasureout} to not only prove which outer automorphisms are random walk generic, but to understand properties held by a typical (with respect to the harmonic measure) tree in the boundary of Outer space.

Outer space is a simplicial complex minus some faces: it has an open simplex for each marked graph (obtained by varying the lengths of the edges of that graph). The faces of a simplex are obtained by collapsing the edges of a forest (making their lengths zero). It is known (via \cite{loneaxes}) that a principal fully irreducible outer automorphism has only a single axis and that this axis is determined by a sequence of ``proper full folds.'' It is further known that it stays in the highest two dimensions of open simplices. It has not previously been known precisely which of these simplices the principal axes pass through. The main results of this paper also determine these simplices in rank 3.

\begin{mainthmB}
The open simplices in $CV_3$ with principal axes passing through them are precisely those whose underlying graph is listed in Main Theorem A, i.e. each of the simplices in the highest 2 dimensions.
\end{mainthmB}

It is worth remarking that the code written determining the automata in this paper could in theory be used to determine the automata in other ranks (computation capacity permitting). The arguments used are also likely extendable to prove results in specific higher-rank settings of interest, with computation time again likely the most significant obstacle.

\subsection*{Acknowledgements}
This work is inspired by results obtained together with Lee Mosher, Yael Algom-Kfir, and Ilya Kapovich. The second author is grateful to their knowledge and insights shared over the years. We thank Martin Lustig for pointing out an important oversight in a previous version and the referee for their helpful suggestions.

\section{Background}
	
	We assume throughout this paper that $r\ge 3$ is an integer, and that $F_r$ is the rank-$r$ free group with a fixed free basis.

\subsection{Edge maps \& train track maps}{\label{ss:EdgeMapsTTMaps}}
	
\begin{df}[High-valence graphs]
		We call a graph $\G$ \textit{high-valence} if the valence of each vertex is $\geq 3$.
\end{df}
	
\begin{df}[Edge, vertex, \& direction sets $E(\G)$, $E^+(\G)$, $V(\G)$, $\mathcal{D}(\G)$, $\mathcal{D}(\G,v)$, \& Turns]
		Let $\G$ be a graph with positively oriented edges $\{e_1, e_2 \dots, e_n\}$ and vertices $\{v_1, v_2, \dots, v_m\}$. Then $E(\G) = \{e_1, \overline{e_1}, \dots, e_n, \overline{e_n}\}$ will denote the set of edges of $\G$ with each orientation listed separately and an overline indicating a reversal of orientation (as it will throughout for both edges and paths). We let $E^+(\G) = \{e_1, e_2 \dots, e_n\}$ denote the set of positively-oriented edges of $\G$ and  $V(\G) = \{v_1, v_2, \dots, v_n\}$ the vertex set of $\G$.
		
A \textit{direction at $v$} will mean an element of $E(\G)$ with initial vertex $v$. We denote the set of directions of $\G$ by $\mathcal{D}(\G)$ and the set of those at $v$ by $\mathcal{D}(\G,v)$.
		
A \textit{turn at $v$} will mean an unordered pair $\{d_1, d_2\}$ of directions $d_1,d_2 \in \mathcal{D}(\G,v)$. The turn is \emph{degenerate} if $d_1=d_2$, and \emph{nondegenerate} otherwise.
	\end{df}
	
	\begin{df}[Paths \& loops] Let $\G$ be a graph. A \textit{path $\rho$ in $\G$} will mean a finite sequence

\noindent $(a_1, a_2, \dots, a_m)$ such that 
\begin{itemize}
  \item[1.] $a_j \in E(\G)$ for each $j$ and
  \item[2.] there exists a sequence $(v_1, v_2, \dots, v_{m+1})$ of vertices in $\G$ satisfying that, for each $i=2,3,\dots,m$, the turn $\{\overline{a_{i-1}}, a_i\}$ is at $v_i$. 
\end{itemize}
\vspace{-0.5em}
\parpic[r]{\includegraphics[width=1in]{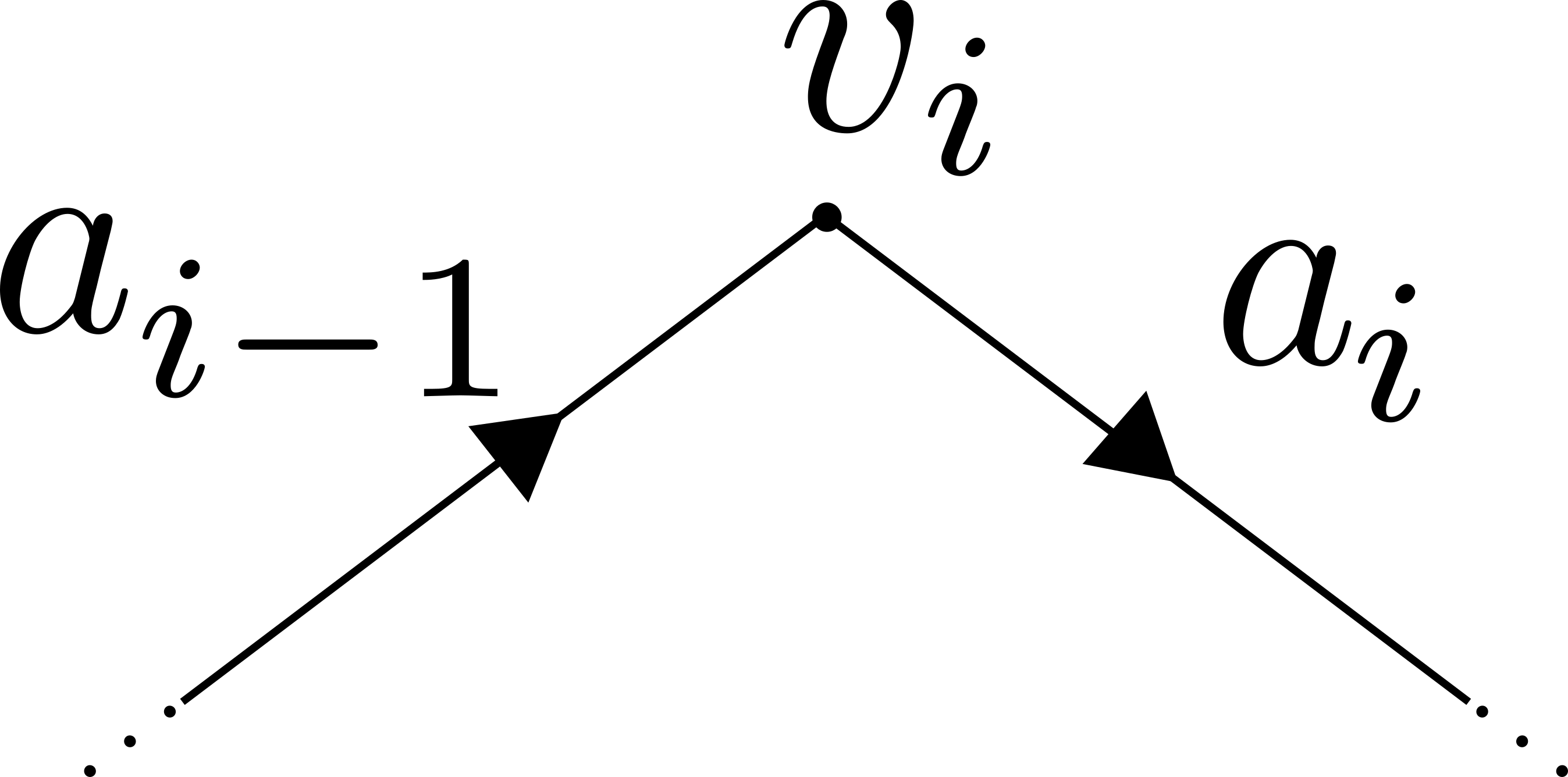}}
For such a path $(a_1, a_2, \dots, a_m)$ we write $\rho = a_1a_2\dots a_m$ and say $\rho$ \textit{contains} the oriented edges $a_1, a_2, \dots, a_m$ and \textit{takes} the turns 
$\{\overline{a_1}, a_2\}$, $\{\overline{a_2}, a_3\}$, $\dots$, $\{\overline{a_{n-1}}, a_n\}$. If $v_1 = v_{m+1}$, we call $\rho$ a \textit{loop in $\G$}. We consider two loops to be equivalent if they are equal up to a cyclic permutation of their edges. Sometimes we ignore orientation in a loop $\ell$, considering $\ell$ and $\overline{\ell}$ the same. In such cases, we call $\ell$ \textit{unoriented}. Otherwise, we call $\ell$ \textit{oriented}. If $\ell$ traverses each edge of $\G$ in at least one of its orientations at least once, we call $\ell$ \textit{comprehensive}. We call both a path and loop $\ell$ \textit{tight} if for each $e\in E(\G)$, the path $e\overline{e}$ does not appear in $\ell$. We may \textit{tighten} $\ell$ by removing, for each $e\in E(\G)$, all instances of $e\overline{e}$ in $\ell$.
\end{df}

\begin{df}[Edge maps, direction map $Dg$]
By an \textit{edge map} $g:\G\to\G'$ we will mean
\begin{itemize}
\item a map $\mV:V(\G)\to V(\G')$, where we write $g(v)$ for $\mV(v)$, together with
\item an assignment, for each $e\in E(\G)$, of a path (denoted $g(e)$) in $\G'$ such that
\begin{enumerate}
\item if the initial vertex of $e$ is $v$, then the initial vertex of $g(e)$ is $g(v)$, and
\item if $g(e)$ is the edge path $g(e) = a_1a_2\dots a_m$, then $g(\overline{e})$ will be the concatenation of edge paths $g(\overline{e}) = \overline{a_m}\dots\overline{a_2}~\overline{a_1}$.
\end{enumerate}
\end{itemize}
If $\gamma = e_1e_2\dots e_n$ is a path in $\G$ for some $e_1, e_2, \dots, e_n \in E(\G)$, then by $g(\gamma)$ we will mean the concatenation of edge paths $g(\gamma) = g(e_1)g(e_2)\dots g(e_n)$.

We call $g$ \emph{tight} if the image of each edge is a tight path. We often, without explicitly writing it, assume edge maps are tight.

Viewing $g$ as a continuous map of graphs, we say $g$ \emph{represents} $\vphi$ when $\pi_1(\G)$ has been identified with $F_r$, i.e. $\G$ is \emph{marked}, and $\vphi$ is the induced map of fundamental groups. When a marking is not explicitly given, we mean that ``there exists a marking such that.''
		
To $g$ we associate a \emph{direction map} $Dg:\mD(\G)\to\mD(\G')$ such that if $g(e) = a_1 a_2\dots a_m$, for some $m\geq 1$ and $a_1,a_2,\dots,a_m \in E(\G)$, then $Dg(e) = p_1$ and $Dg(\overline{e}) = \overline{a_m}$. We call a direction $e$ \emph{periodic} if $Dg^k(e) = e$ for some $k>0$, and \emph{fixed} if $k=1$.
\end{df}

\begin{rmk}
		Let $\rho$ be a path in a graph $\G$ and $g$ an edge map with domain $\G$. Then $g(\rho)$ is tight if and only if $g$ can be viewed as a continuous map that is locally injective on $\rho$.
	\end{rmk}
	
\begin{df}[Interior \& exterior turns]\label{d:IntExtTurns}
Suppose $g:\Gamma\to\Gamma'$ is an edge map, $\gamma$ a path in $\Gamma$, and $\gamma' \coloneqq g(\gamma)$. The \textit{interior turns of $g$ in $\Gamma'$} will be the turns taken by the paths $g(E)$, such that $E\in E^+(\Gamma)$. The \textit{interior turns of $g$ in $\gamma'$} will be the turns taken by the paths $g(E)$ such that $E\in E(\Gamma)$ and $\gamma$ contains $E$. If $\{a,b\}$ is a turn in $\gamma$, then the turn $\{Dg(a), Dg(b)\}$ in $\gamma'$ is called an \textit{exterior turn of $g$ in $\gamma'$}.
\end{df}

\begin{df}[Illegal turns \& transparent edge maps]\label{d:Transparent} Let $g \colon \Gamma \to \Gamma'$ be an edge map. We call a turn $\{d_1,d_2\}$ in $\Gamma$ g-\emph{prenull} if $\{Dg(d_1),Dg(d_2)\}$ is degenerate.
When $g\from \Gamma \to \Gamma$ is a self-map, the turn $\{d_1,d_2\}$ is called an \emph{illegal turn} for $g$ if $\{Dg^k(d_1),Dg^k(d_2)\}$ is degenerate for some $k$.
 We call $g$ \emph{transparent} if each illegal turn is prenull, if each turn taken by $g^k(e)$ for some $k\in\mathbb{N}$ and $e\in E(\G)$ is taken by $g(e)$, and if each $g$-periodic vertex and direction is in fact fixed by $g$. Notice that every self-map has a transparent power. This is a strengthening of the notion given in \cite{stablestrata} and so in particular implies it.
\end{df}

\begin{lem}{\label{l:dirimg}}
		Suppose $F:\Gamma\to\Gamma'$ is an edge map, $\gamma$ a path in $\Gamma$, and $\gamma' \coloneqq F(\gamma)$. Then
\begin{enumerate}
			\item if $\gamma$ takes the turn $\{a,b\}$ in $\Gamma$, then $\gamma'$ takes the turn $\{DF(a), DF(b)\}$ in $\Gamma'$, and
			\item each turn in $\gamma'$ is either an interior or an exterior turn of $F$ in $\gamma'$ before tightening.
\end{enumerate}
\end{lem}
	
\begin{proof} Suppose that $F:\Gamma\to\Gamma'$ is as in the lemma statement.

{\bf{(1)}}
 Let $a_1$, $\dots$, $a_n$, $b_1$, $\dots$, $b_m$ be oriented edges of $\Gamma'$ such that $F(a) \coloneqq a_1\dots a_n$ and $F(b) = b_1\dots b_m$.
			Suppose further that $\gamma$ takes $\{a,b\}$. Then $\gamma$, up to reorientation, contains $\overline{a}b$. So $\gamma'$, also up to reorientation, contains $F(\overline{a})F(b) = (\overline{a_n})(\overline{a_{n-1}}) \dots(\overline{a_1})b_1b_2\dots b_m$. In particular, $\gamma'$ takes the turn $\{a_1, b_1\} = \{DF(a), DF(b)\}$ in $\Gamma'$. This completes the proof of (1).
			
			{\bf{(2)}} Suppose $\gamma = e_1e_2\dots e_n$, where $e_1$, $e_2$, $\dots$, $e_n \in E(\G)$, and $F(e_i) = e_{i,1}e_{i,2}\dots e_{i,k_{i}}$.	Then,
			\[\gamma' = F(e_1)F(e_2)\dots F(e_n) = e_{1,1}e_{1,2}\dots e_{1,k_1}e_{2,1}e_{2,2}\dots e_{2,k_2} \dots e_{n,1}e_{n,2}\dots e_{n, k_n}.\]
			Observe that each turn in $\gamma'$ is either of the form $\{\overline{e_{i,j}}, e_{i,j+1}\}$ or $\{\overline{e_{i,k_i}}, e_{i+1, 1}\}$.
			A turn of the form $\{\overline{e_{i,j}}, e_{i,j+1}\}$ is in the image $F(e_i)$ and is therefore an interior turn of $F$ in $\gamma'$.
			For a turn of the form $\{\overline{e_{i,k_i}}, e_{i+1, 1}\}$, we have $\{\overline{e_{i,k_i}}, e_{i+1, 1}\} = \{\overline{DF(e_i)}, DF(e_{i+1})\}$, where $\{\overline{e_i}, e_{i+1}\}$ is a turn taken by $\gamma$. Thus, a turn of the form $\{\overline{e_{i,k_i}}, e_{i+1, 1}\}$ is an exterior turn of $F$ in $\gamma'$. This completes the proof of (2).\\
\end{proof}

\subsection{Folds \& (fold-conjugate) Stallings fold decompositions}{\label{ss:Folds}}

\begin{df}[(Different-length) \& (permissible) folds]{\label{d:folds}}
Let $\G$ be a graph and $\{a,b\}$ a nondegenerate turn in $\G$. We define the different-length fold $\mathfrak{f}$ in $\G$ of $\{a,b\}$ with $a$ longer as follows.
\parpic[r]{\includegraphics[width=2in]{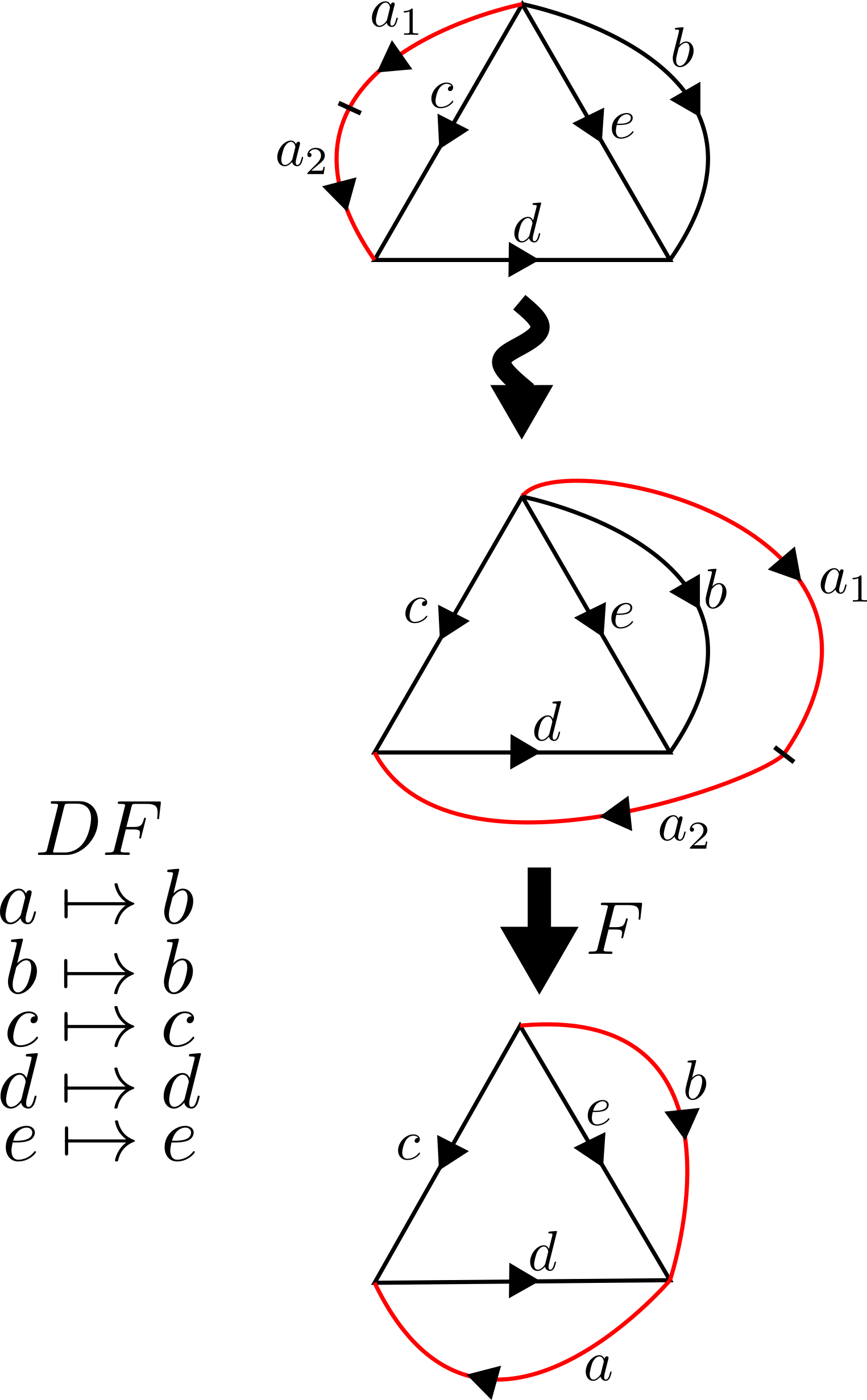}}
		If we subdivide the edge $a$ into $a_1$ and $a_2$, so $a=a_1 a_2$, then by the \textit{different-length fold $\mathfrak{f}$ in $\G$ of $\{a,b\}$ with $a$ longer} we will mean the quotient map of $\G$ defined by identifying the edges $a_1$ and $b$. Note that for the different-length fold $\f$ of $\{a,b\}$ with $a$ longer, if we label the edges in $\f(\G)$ so that $\f(a) = b'a'$, and for all other $e\in E^+(\Gamma)$ with $e\neq a$, we have $\f(e) = e'$, then the direction map will fix all directions apart from sending $a$ to $b'$. A different-length fold is sometimes elsewhere in the literature referred to as a ``proper full fold,'' but we found our terminology more evocative in that, in Culler-Vogtmann Outer space, the different-length fold of $\{a,b\}$ with $a$ longer could only occur if $a$ were longer than $b$. At times we abuse notation and drop the primes.

		Let $\G$ be a high-valence graph, $\{a,b\}$ a turn in $\G$, and $\ell$ a loop in $\G$. We say the different-length fold of $\{a,b\}$ with $a$ longer in $\G$ is \textit{permissible with respect to $\ell$} if $\ell$ does not take the turn  $\{a,b\}$.
	\end{df}
	
\begin{rmk}
		Let $\G$ be a high-valence graph,  $\ell$ a comprehensive loop in $\G$, and $\f$ a permissible fold of $\G$ with respect to $\ell$. Let $\G' \coloneqq \f(\G)$ and $\ell' \coloneqq \f(\ell)$. Since $\f$ is a quotient map, and hence surjective, $\ell'$ is also comprehensive.
	\end{rmk}

\begin{df}[(Fold-conjugate) Stallings fold decompositions]{\label{d:stallingsfolds}}
	Stallings introduced a version of a fold in \cite{s83}. Let $g \colon \Gamma \to \Gamma'$ be a tight graph homotopy equivalence. Let $e_1' \subset e_1$ and $e_2' \subset e_2$ be maximal, initial, nontrivial subsegments of edges $e_1$ and $e_2$ emanating from a common vertex and satisfying: $g(e_1')=g(e_2')$ as edge paths and the terminal endpoints of $e_1'$ and $e_2'$ are distinct points in $g^{-1}(V(\Gamma))$. Redefining $\Gamma$ having vertices at the endpoints of $e_1'$ and $e_2'$ if necessary, one can obtain a graph $\Gamma_1$ by identifying the points of $e_1'$ and $e_2'$ that have the same image under $g$. We call such a fold a \emph{Stallings fold}. Stallings \cite{s83} also showed that if $g \colon \Gamma \to \Gamma'$ is a tight homotopy equivalence, then $g$ factors as a composition of Stalling folds and a final homeomorphism. We call such a decomposition a \emph{Stallings fold decomposition}. It can be obtained as follows: At an illegal turn for $g\colon \Gamma  \to \Gamma'$, one can fold two maximal initial segments having the same image in $\Gamma'$ to obtain a map $\mathfrak{g}_1 \colon \Gamma_1 \to \Gamma'$ of the quotient graph $\Gamma_1$. The process can be repeated for $\mathfrak{g}_1$ and recursively. If some $\mathfrak{g}_k \colon \Gamma_{k-1} \to \Gamma$ has no illegal turn, then $\mathfrak{g}_k$ will be a homeomorphism and the fold sequence is complete.

\[
\xymatrix{\Gamma_0 \ar[r]_{g_1} \ar@/^4pc/[rrrr]_{g=\mathfrak{g}_0} & \Gamma_1 \ar[r]_{g_2} \ar@/^3pc/[rrr]_{\mathfrak{g}_1} & \Gamma_2 \ar[r]_{g_3} \ar@/^2pc/[rr]_{\mathfrak{g}_2} & \dots \ar[r]_{g_n}  & \Gamma_n=\Gamma' \\}
\]

Notice that choices of illegal turns are made in this process and that different choices lead to different Stallings fold decompositions of the same homotopy equivalence.

A \emph{subdivided fold} will mean a fold written as a composition of 2 folds. Suppose $\G_{0} \xrightarrow{g_1} \G_{1} \xrightarrow{g_2} \cdots \xrightarrow{g_{k-1}} \G_{k-1} \xrightarrow{g_k} \G_k$ and $\G'_{0} \xrightarrow{h_1} \G'_{1} \xrightarrow{h_2} \cdots \xrightarrow{h_{n-1}} \G'_{n-1} \xrightarrow{h_n} \G'_n$ are Stallings fold decompositions. We say the decompositions are \emph{fold-conjugate} if, possibly after a fold subdivision of $g_k$ into $\G_{k-1} \xrightarrow{g_k'} \G_{k'} \xrightarrow{g_{k+1}} \G_k$ or $h_n$ into $\G'_{n-1} \xrightarrow{h_n'} \G'_{n'} \xrightarrow{h_{n+1}} \G'_n$, we have for some $j$:\\

\noindent $\G_j \xrightarrow{g_{j+1}} \cdots \xrightarrow{g_k} \G_k=\G_{0} \xrightarrow{g_1} \cdots \xrightarrow{g_j} \G_j $ is
$\G'_{0} \xrightarrow{h_1} \G'_{1} \xrightarrow{h_2} \cdots \G'_{n-1} \xrightarrow{h_n'} \G'_{n'} \xrightarrow{h_{n+1}} \G'_n$ or \\

\noindent  $\G_j \xrightarrow{g_{j+1}} \cdots \xrightarrow{g_k} \G_k=\G_{0} \xrightarrow{g_1} \cdots \xrightarrow{g_j} \G_j $ is
$\G'_{0} \xrightarrow{h_1} \G'_{1} \xrightarrow{h_2} \cdots \xrightarrow{h_{n-1}} \G'_{n-1} \xrightarrow{h_n} \G'_n$ or
\\

\noindent $\G_{0} \xrightarrow{g_1} \cdots \xrightarrow{g_{k-1}} \G_{k-1} \xrightarrow{g_k'} \G_{k'} \xrightarrow{g_{k+1}} \G_k$ is
$\G'_j \xrightarrow{h_{j+1}} \cdots \xrightarrow{h_n} \G'_n=\G'_{0} \xrightarrow{h_1} \cdots \xrightarrow{h_{j}} \G'_{j} $.\\
\end{df}

The following lemma will be used in Section \ref{s:LonelyDirectionProperty} in understanding the structure of the local Whitehead graph for a train track representative of a principal fully irreducible outer automorphism.
	
\begin{lem}{\label{l:foldturns}}
		Let $\Gamma$ be a graph, $\overline{a}$ and $b$ directions at a common vertex of $\Gamma$, and $\ell$ a comprehensive loop in $\Gamma$. Let $\f$ be the different-length fold of $\{\overline{a},b\}$ with $b$ longer. We use the notation for the image of $\f$ established in Definition~\ref{d:folds}: i.e. $\f(b) = \overline{a'}b'$, and for all other $e\in E^+(\Gamma)$ with $e\neq b$, we have $\f(e) = e'$. Let $\Gamma'$ and $\ell'$ denote the $\f$-images of $\Gamma$ and $\ell$ respectively. Then, before tightening,  $\ell'$ takes the turn $\{a',b'\}$ and $\ell'$ takes no other turns containing $b'$.
\end{lem}
	
\begin{proof}
Since $\ell$ is a comprehensive loop, $\ell$ contains $b$. So $\ell'$ contains $\f(b) = \overline{a'}b'$ and thus, pre-tightening, $\ell'$ takes the turn $\{a',b'\}$. Now suppose $\{d_1, d_2\}$ is a turn taken by $\ell'$ pre-tightening. By Lemma~\ref{l:dirimg}, we have that $\{d_1, d_2\}$ is either an interior or exterior turn of $\f$ in $\ell'$. If $\{d_1, d_2\}$ is an interior turn, it must be $\{a', b'\}$, as the quotient map fixes all edges but $b$. So we assume $\{d_1, d_2\}$ is an exterior turn. Since there is no $e\in E(\G)$ such that $D\f(e) = b'$, we then have that $\{d_1, d_2\}$ cannot contain the direction $b'$, completing the proof.\\
\end{proof}

\section{The Rank-3 Lonely Direction Train Track Automaton and Map $F$}

In this section we define the automaton satisfying that each Stallings fold decomposition of each train track representative of each principal fully irreducible outer automorphism is represented by a directed loop in the automaton (see \S \ref{s:tt} and \S \ref{s:MainProof}).\\

\subsection{Sequences of permissible folds and the automaton $\mA(\G,\ell)$ generated by $(\G,\ell)$}
	
	\begin{df}[Sequence of permissible folds]
		Let $\G$, $\G_1$, $\G_2$, $\dots$, $\G_n$ be rank-$r$ high-valence graphs and let $\ell$ be a comprehensive loop in $\G$. Let $\f_1:\G\rightarrow \G_1$,  $\f_2:\G_1\rightarrow \G_2$, $\dots$, $\f_n:\G_{n-1}\rightarrow \G_n$ be folds and $\ell_1 \coloneqq \f_1(\ell)$, $\ell_2 \coloneqq \f_2(\ell_1)$, $\dots$, $\ell_n \coloneqq \f_n(\ell_{n-1})$.
		Suppose $\f_i$ is a permissible fold in $\G_{i-1}$ with respect to $\ell_i$ for each $i$. We then call $\f = (\f_1, \f_2, \dots, \f_n)$ a \textit{sequence of permissible folds in $\G$}, from which we obtain $\G_n$. Abusing notation, we let $\f$ also denote $\f_n \circ \f_{n-1} \circ \dots \circ \f_1$.
	\end{df}

	\begin{lem}{\label{F(ell) tight}}
		Let $\G$ and $\G'$ be high-valence graphs, $\ell$ and $\ell'$ comprehensive loops in $\G$ and $\G'$ respectively, and $\f:\G\to\G'$ a permissible different-length fold with respect to $\ell$ such that $\f(\ell) = \ell'$. Then if $\ell$ is tight, $\ell'$ is tight.
	\end{lem}

	\begin{proof}
		Denote by $\{a, \overline{b}\}$ the turn folded by $\f$ and assume, without loss of generality, that $a$ is longer.
		We use the notation for the image of $\f$ established in Definition~\ref{d:folds}. That is $\f(a) = \overline{b'}a'$, and for each other $e\in E^+(\Gamma)$ with $e\neq b$, we have $\f(e) = e'$. Let $\{d_1', d_2'\}$ be an arbitrary turn taken by $\ell'$. For $\ell'$ to be tight it suffices to show $d_1'\neq d_2'$. By Lemma~\ref{l:dirimg},
		$\{d_1', d_2'\}$ is either an interior or exterior turn of $\f$. We examine the cases separately:\\

		\vspace{-0.9em}
		{\bf{Case 1 ($\{d_1', d_2'\}$ is an interior turn):}}
		Since all edges in $E^+(\G)$ except $a$ are fixed by $\f$ (but for the addition of primes), the interior turn $\{d_1', d_2'\}$ must be taken by $\f(a) = \overline{b'}a'$. Thus, $\{d_1', d_2'\} = \{a',b'\}$. Thus, it suffices to show $a' \neq b'$.
		
		Suppose for the sake of contradiction that $a' = b'$. Subdivide $a$ into $a_1a_2$ with $\f(a_1) = \overline{b'}$ and $\f(a_2) = a'$. Then $\f^{-1}(\{a'\}) = \{a_2\}$ and $\f^{-1}(\{\overline{b'}\}) = \{a_a, \overline{b}\}$. By construction, $a_1 \neq a_2$. Moreover, $\f$ is a fold of $a = a_1a_2$ and $\overline{b}$, so $a_1$ and $\overline{b}$  must be distinct directions. Thus, $a' \neq b'$, concluding the proof in the case where $\{d_1', d_2'\}$ is an interior turn.\\

		\vspace{-0.9em}
		{\bf{Case 2 ($\{d_1', d_2'\}$ is an exterior turn):}}
		If $\{d_1', d_2'\}$ is an exterior turn, then $\{d_1', d_2'\} = \{D\f(d_1), D\f(d_2)\}$ for some turn $\{d_1, d_2\}$ taken by $\ell$. We examine subcases:\\

		\vspace{-0.9em}
		{\bf{Case 2a ($d_1 \neq a$ and $d_2\neq a$):}}
		Since $\ell$ is tight, $d_1 \neq d_2$. Since $\f$ fixes all directions except for $a$ (but for the addition of primes), the directions $d_1$ and $d_2$ must also be fixed under $\f$ (but for the addition of primes). Thus, $D\f(d_1) \neq D\f(d_2)$. Since $\{d_1', d_2'\} = \{D\f(d_1), D\f(d_2)\}$, we have $d_1' \neq d_2'$, as desired.\\

		\vspace{-0.9em}
		{\bf{Case 2b ($d_1 = a$ or $d_2 = a$):}}
		Without loss of generality, we assume $d_1 = a$. Then $d_1' = D\f(d_1)=D\f(a) = \overline{b'}$. To show $d_1' \neq d_2'$, it suffices to show $d_2' = D\f(d_2) \neq \overline{b'}$.
		
		Since $d_2 \in D\f^{-1}(\{d_2'\})$, it suffices to show $d_2\notin D\f^{-1}(\{\overline{b'}\})$. Suppose for the sake of contradiction that $d_2 \in D\f^{-1}(\{\overline{b'}\})$. Now $D\f^{-1}(\{\overline{b'}\}) = \{a, \overline{b}\}$. Further, since $\ell$ is tight and $\{d_1, d_2\}$ is taken by $\ell$ and $d_1 = a$, we have $d_2 \neq a$. So, if we had $d_2 \in \{a, \overline{b}\}$, we would need $d_2 = b$. So suppose $d_2 = b$. Then $\ell$ would take the turn $\{d_1, d_2\} = \{a, \overline{b}\}$. But $\f$ is a fold of $\{a,\overline{b}\}$ and is permissible with respect to $\ell$, a contradiction. Thus, $d_2 \neq \overline{b}$, as desired.\\
	\end{proof}

	\begin{lem}{\label{l:F local inj}}
		Let $\G$ be a high-valence graph, $\ell$ a tight comprehensive loop in $\G$, and $\f = (\f_1, \f_2, \dots, \f_n)$ a sequence of permissible folds in $\G$. Then for each $e\in E(\G)$, we have that $\f(e)$ is tight.
	\end{lem}

	\begin{proof}
		Since $\ell$ is comprehensive, $\ell$ contains $e$. Applying $\f$ to $\ell$ consists of successively applying folds to $\ell$. Since, by Lemma~\ref{F(ell) tight}, a permissible fold with respect to a loop preserves its tightness, it follows that $\f(\ell)$ is tight. Therefore, $\f(e)$ is tight.\\
	\end{proof}

	\begin{df}[Automaton $\mA(\G,\ell)$ generated by $(\G,\ell)$]
		Let $\G$ be a high-valence graph, $\ell$ a comprehensive loop in $\G$, and $\mF$ the set of all sequences of permissible folds in $\G$. Let
		\[\mG = \{(\G_i, \ell_i) \mid \G_i = \f(\G), \ell_i = \f(\ell)\text{ for some } \f \in \mF\}\]
		be the set of tuples obtained by applying all possible sequences of folds to $\G$ and $\ell$. Let $\mE \subseteq \mG^2$ be the set such that $((\G_i, \ell_i),(\G_i', \ell_i')) \in \mE$ precisely when there exists a permissible fold $\f: \G_i \to \G_i'$ with respect to $\ell_i$ such that the set of turns taken by $\ell_i'$ is equal to the set of turns taken by $\f(\ell_i)$, modulo primes.
		
		We consider two elements $(\G_i, \ell_i)$ and $(\G_j, \ell_j)$ of $\mG$ equivalent if the following are satisfied:
		\begin{itemize}
			\item there exists an ornamentation-preserving graph isomorphism $\sigma$ from $\G_i$ to $\G_j$ and
			\item the set of turns taken by $\ell_i$ can be identified via $\sigma$ with the set of turns taken by $\ell_j$.
		\end{itemize}
		
		Then \textit{the automaton generated by $(\G,\ell)$}, denoted $\mathcal{A}(\G,\ell)$, is the directed graph $(\mG, \mE)$.\\
	\end{df}

\subsection{The Lonely Direction Property}\label{s:LonelyDirectionProperty}

We introduce in this section a property held by the states (and preserved within) the automata we construct.
	
	\begin{prop}
		Let $\G$ be a $5$-edged rank-$3$ high-valence graph. Then $\G$ has exactly 3 vertices: two valence-3 vertices and one valence-4 vertex.
	\end{prop}
	
	\begin{proof}
		Note that all vertices of $\G$ must have valence $\geq 2$, and that, since there are $5$ edges, the sum of the valences of all the vertices must be $10$.
		
		Suppose $\G$ were a rank-$3$ graph with $5$ edges and $1$ vertex. This would be a rank-$5$ rose, contradicting that $\G$ has rank $3$.
		
		Now suppose $\G$ was a rank-$3$ graph with $5$ edges and $2$ vertices. One may partition the $5$ edges into edges that connect the $2$ vertices and edges that form loops on one of the vertices. The graph with $2$ vertices and $5$ edges connecting the vertices has rank $4$. Since replacing an edge connecting the vertices with an edge forming a loop at a vertex does not change the rank, $\G$ also cannot have this form.
		
		Now suppose $\G$ was a rank-$3$ graph with $5$ edges and $\geq 4$ vertices. Then, by the valence sum, there would be a vertex with valence $< 3$, but we said $\G$ had to be of high valence.
		
		Thus, $\G$ must have precisely $3$ vertices. Since each vertex must have valence $\geq 3$, the valences of the vertices must be $3$, $3$, and $4$.\\
	\end{proof}

	\begin{df}[Lonely Direction Property]
		Let $\G$ be a $5$-edged rank-$3$ graph, $\ell$ a comprehensive loop in $\G$ taking all but two turns of $\G$, and $v$ the valence-$4$ vertex of $\G$. If there exists a particular direction $d$ at $v$ that is contained in the only two turns not taken by $\ell$, then we say $(\G, \ell)$ satisfies the \textit{Lonely Direction Property with distinguished direction $d$}.
	\end{df}

	\begin{prop}\label{p:LoneDirectionPropertyPreserved}
		Let $\G$ be a $5$-edged rank-$3$ graph, and $\ell$ a comprehensive loop in $\G$ such that $(\G, \ell)$ satisfies the \LDP. Then, for all $5$-edged graphs $\G_i$ in $\mA(\G,\ell)$, letting $\ell_i$ denote image of $\ell$ in $\G_i$, we have that $(\G_i, \ell_i)$ satisfies the \LDP.
	\end{prop}
	
	\begin{proof}
		Let $v$ be the valence-4 vertex of $\G$ and $u,w$ the valence-3 vertices. Note that all permissible folds in $\G$ are at $v$, since the only two turns not taken by $\ell$ are at $v$. Suppose that $\f\colon \G \to \G'$ is a permissible fold of the turn $\{a,b\}$ at $v$ with $b$ longer and that either $a$ or $b$ is the distinguished direction. As in Definition~\ref{d:folds}, subdivide $b$ into $b_1b_2$ so that $\f(a) = \f(b_1)$. Call $\f(a)$ by $a'$ (so that also $\f(b_1) = a'$) and $\f(b)$ by $b'$. Then $D\f(a) = D\f(b) = a'$. Let $\f(\ell):=\ell '$ and let $v'$, $u'$, and $w'$ denote the respective $\f$-images of $v$, $u$, and $w$.
		
		The valence of $v'$ will be either 3 or 4. Specifically, the valence of $v'$ is 4 if $a$ is a loop and is 3 otherwise. We consider separately:

		\vskip 5pt
				
		{\bf{Case 1:}} $v'$ has valence 3~~~and~~~{\bf{Case 2:}} $v'$ has valence 4.
		
		\vskip 5pt
		
		\noindent {\bf{Case 1 ($v'$ has valence 3):}} We aim to show that the valence-4 vertex of $\G'$ contains all turns not taken by $\ell'$, and each of these turns contains a common distinguished direction.
	\parpic[r]{\includegraphics[width=3in]{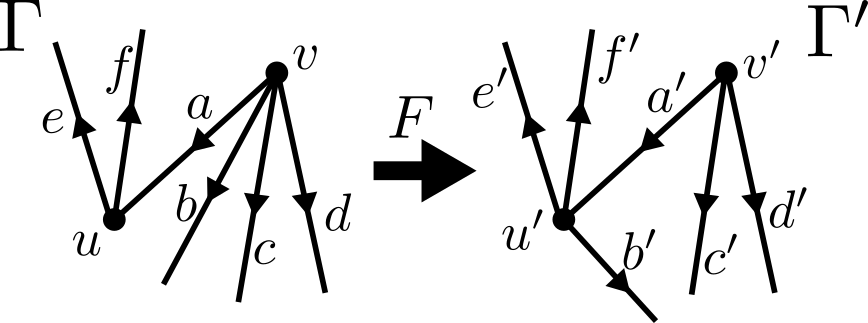}}	
		Suppose the directions at $v$ are $\{a,b,c,d\}$, i.e. $c$ and $d$ are the $2$ directions not involved in the fold. Then letting $c' \coloneqq D\f(c)$ and $d' \coloneqq D\f(d)$, the directions at $v'$ are $\{a',c',d'\}$.
		Then $\ell '$ takes all turns at $v'$, as follows. By Lemma~\ref{l:dirimg}, $\ell '$ takes $\{a', c'\}$ and $\{a',d'\}$ as follows. Either $\ell$ took $\{b,c\}$ and $\{b,d\}$ if $a$ was the distinguished direction or $\{a,c\}$ and $\{a,d\}$ if $b$ was the distinguished direction. Moreover, $\ell '$ takes $\{c',d'\}$, because $\ell$ took $\{c,d\}$.
		
		Suppose, without loss of generality, that the vertex $u$ of $\G$ contains the direction $\bar{a}$. Note that $v$ cannot contain $\bar{a}$, or $a$ would be a loop and $v'$ would have valence 4.
		Let $e$ and $f$ denote the other directions at $u$. Let the images of $u$, $\bar{a}$, $e$, and $f$ under $\f$ be $u'$, $\bar{a}'$, $e'$, and $f'$ respectively. The directions at $u'$ are then $\{\bar{a}', e', f', b'\}$ (see image above), where $b'$ is the image of the portion of $b$ not identified with $\bar{a}$, i.e. $b_2$.
		
		By Lemmas~\ref{l:dirimg} and \ref{l:foldturns}:	
		
		$\bullet$ Since $\ell$ takes all turns at $u$, we have $\ell '$ takes all turns among $\{\bar{a}', e', f'\}$ and
		
		$\bullet$ $\ell'$ takes the turn $\{\bar{a}',b'\}$ and
		
		$\bullet$ $\ell'$ takes no other turns at $u'$ involving $b'$.
		
		\noindent That is, all turns at $u'$ that $\ell'$ does not take involve $b'$.
		
		Also by Lemma~\ref{l:dirimg}, since $\ell$ takes all turns at $w$ and the directions at $w$ are not involved in the fold $\f$, we have that $\ell'$ takes all turns at $w'$.
		
		This completes the proof of Case 1 with $u'$ as the valence-4 vertex of $\G'$ and $b'$ as the distinguished direction.
		
		\vskip 5pt
		
		\noindent {\bf{Case 2 ($v'$ has valence 3):}}  We aim to show that the valence-4 vertex of $\G'$ contains all of the turns not taken, and each of these turns contains a common distinguished direction.\\
\vspace{-1.4em}
\parpic[r]{\includegraphics[width=1in]{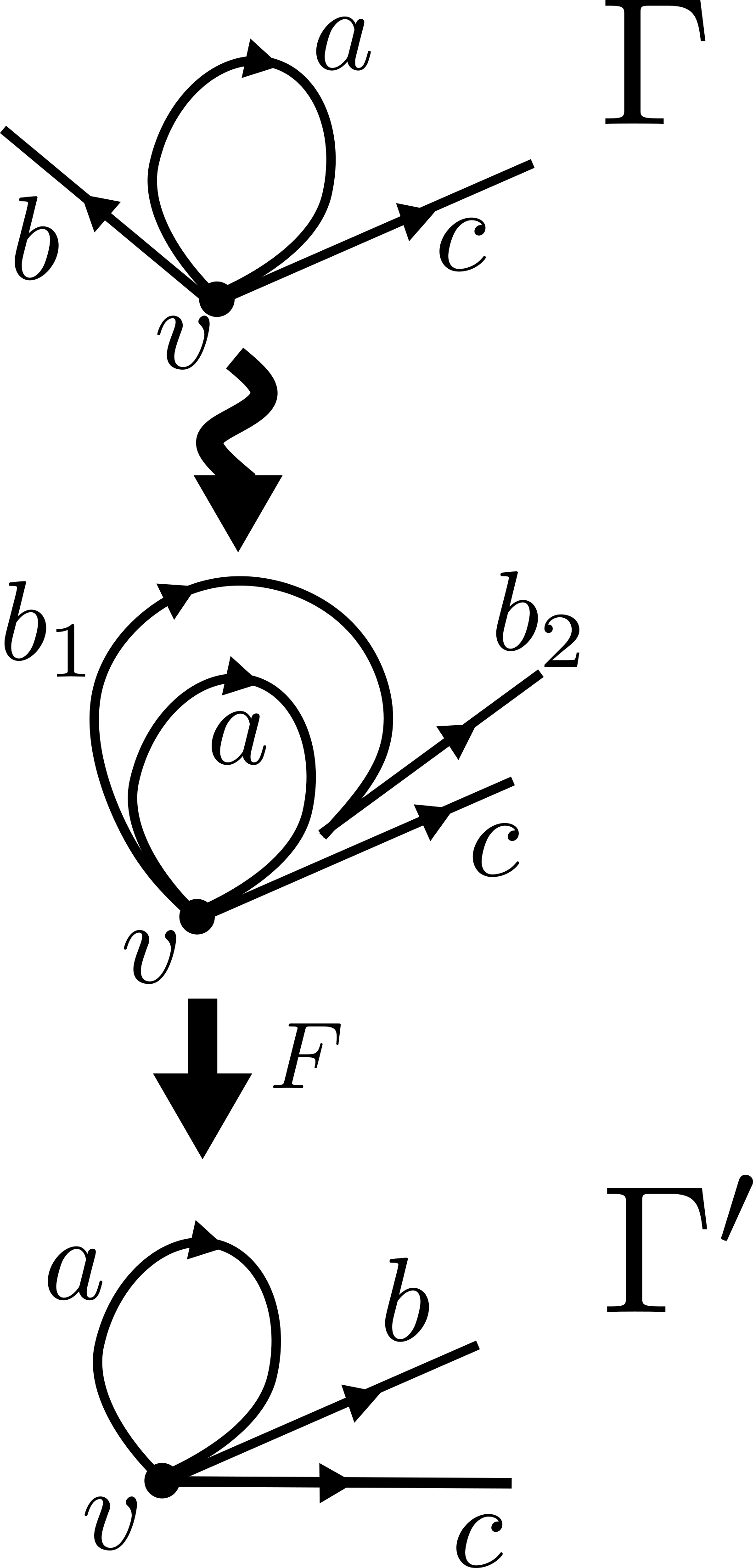}}		
		In this case, the edge $a$ must be a loop at $v$. Note that the $\f$-image of $a$ will also be a loop. We let $c$ denote the remaining direction at $v$ and $c'$ denote $\f(c)$. At $v$, $\ell$ either takes $\{b,\bar{a}\}$, $\{b,c\}$, and $\{\overline{a}, c\}$, if $a$ is the distinguished direction, or $\{a,\bar{a}\}$, $\{a,c\}$, and $\{\overline{a}, c\}$, if $b$ is the distinguished direction.
		Thus, by Lemma~\ref{l:dirimg}, $\ell'$ takes $\{a', \overline{a'}\}$, $\{a', c'\}$, and $\{\overline{a'}, c'\}$. That is, $\ell'$ takes all turns among the directions $\{a', \overline{a'}, c'\}$. Moreover, by Lemma~\ref{l:foldturns}, $\ell'$ takes the turn $\{\overline{a'}, b'\}$ and no other turns at $v'$ involving $b'$. That is, all turns that $l'$ does not pass over at $v'$ involve $b'$.
		Further, since $u$ and $w$ were not involved in any folds under $\f$ and $\ell$ takes all turns at $u$ and $w$, by Lemma~\ref{l:dirimg}, all turns at $u'$ and at $w'$ are taken by $\ell'$. This completes Case 2 with $v'$ as the valence-4 vertex of $\G'$ and $b'$ as the distinguished direction.
		
		Inductively, all 5-edged graphs $\G_i$ in $\mA(\G,\ell)$ must also have the same structure as $\G'$. In particular, $\G_i$ must have a distinguished direction $b$ at the valence-4 vertex such that the image of $\ell$ takes all turns in $\G_i$ except for precisely two turns involving $b$.\\
	\end{proof}

\subsection{Rank-$3$ lonely direction automaton algorithms \& construction}\label{ss:LDAalgs}

The rank-$3$ lonely direction automaton $\mA_3$ is constructed as follows (the code can be found at \cite{g20b}).

		\vskip 5pt

To more generally construct $\mA(\G,\ell)$ for a high-valence graph $\G$ and comprehensive loop $\ell$ in $\G$:
\begin{enumerate}
  \item For each nondegenerate turn $\{d_{i,1},d_{i,2}\}$ in $\G$ such that $\ell$ does not take $\{d_{i,1},d_{i,2}\}$,
	let $\mF_i$ denote the set containing the different-length fold $\{d_{i,1},d_{i,2}\}$ with $d_{i,1}$ longer and the different-length fold $\{d_{i,1},d_{i,2}\}$ with $d_{i,2}$ longer.
	Note that, letting $\mF = \cup_{i} \mF_i$, we have that $\mF$ is the set of all different-length permissible folds of $\G$ with respect to $\ell$.
  \item Let $S_0 \coloneqq (\G, \ell)$.
  \item Given $S_k$, use (1) with inputs $\G_k$ and $\ell_k$ to determine all permissible folds $\mF^k = (\f_{k,1},\f_{k,2},\dots,\f_{k,n_k})$ on $\G_k$ with respect to $\ell_k$.
  \item Perform each fold $\f_{k,j}$ on $\G_k$ and $\ell_k$ to obtain $(\f_{k,j}(\G_k), \f_{k,j}(\ell_k))$. Let
      $$E_{k,j} = ((\G_k, \ell_k), (\f_{k,j}(\G_k), \f_{k,j}(\ell_k))) \text{ and } S_m = (\f_{k,j}(\G_k), \f_{k,j}(\ell_k)),$$
       where $m$ is one more than the maximum of the indices $k$ used thus far.
  \item Let $S = \cup_k S_k $ and $E = \cup_{k,j} E_{k,j}$.
\end{enumerate}
It is not difficult to see that the directed graph $(S, E)_{(\G, \ell)}$ is $\mA(\G, \ell)$.\\

\begin{df}[The rank-$3$ lonely direction automaton $\mA_3$]
We define $\mA_3$ is the maximal strongly connected connected components of $\cup \mA(\G, \ell)$ as $(\G, \ell)$ varies over all pairs $(\G, \ell)$ such that $\G$ is a high-valence graph possessing the Lonely Direction Property with respect to the comprehensive loop $\ell$.
\end{df}

\begin{rmk}
$\mA_3$ is obtained by taking $\cup (S,E)_{(\G, \ell)}$ for suitable $(\G, \ell)$ as above and then taking the maximal strongly connected components.\\
\end{rmk}

\subsection{The rank-$3$ lonely direction automaton}\label{ss:LDA}

Consider $\mG$ to be the set of graphs in the two representations below, but with the inner embedded (colored) graphs replaced with vertices. Each graph $\G_i\in\mG$ has associated with it a comprehensive loop $\ell_i$ such that $\ell_i$
\begin{itemize}
\item takes each turn at each valence-$3$ vertex of $\Gamma_i$ and
\item takes a turn $\{d_1, d_2\}$ at the valence-$4$ vertex of $\G_i$ if and only if $\{d_1, d_2\}$ is an edge on the (colored) graph depicted in the valence-$4$ vertex of $\G_i$ in the figures.
\end{itemize}

Then, $(\G_i, \G_j)$ is an edge in the figures if and only if the following are satisfied:\\
$\bullet$ There exists a permissible fold $\f_{i,j}$ such that $(\f_{i,j}(\G_i), \f_{i,j}(\ell_i))$ is equivalent to $(\G_j, \ell_j)$ up to a permutation of the edge labels. In this case, we write $(\f_{i,j}(\G_i), \f_{i,j}(\ell_i)) \sim (\G_j, \ell_j)$.\\
$\bullet$ There exists a sequence of permissible folds $\mF_{j,i}$ such that $(\f_{j,i}(\G_i), \f_{j,i}(\ell_i)) \sim (\G_j, \ell_j)$.

That is, starting with all high-valence, 5-edged graphs $\G$ and comprehensive loops $\ell$ such that $(\G,\ell)$ satisfies the \LDP, the figures below give the maximal strongly connected components of $\cup \mA(\G,\ell)/\sim$, computed as in \S \ref{ss:LDAalgs} using the program \cite{g20a}.

\vskip10pt

The following 3 graphs give single-graph maximal strongly connected components:

\vskip10pt

\begin{figure}[ht!]\label{f:automaton}
\includegraphics[width=1.5in]{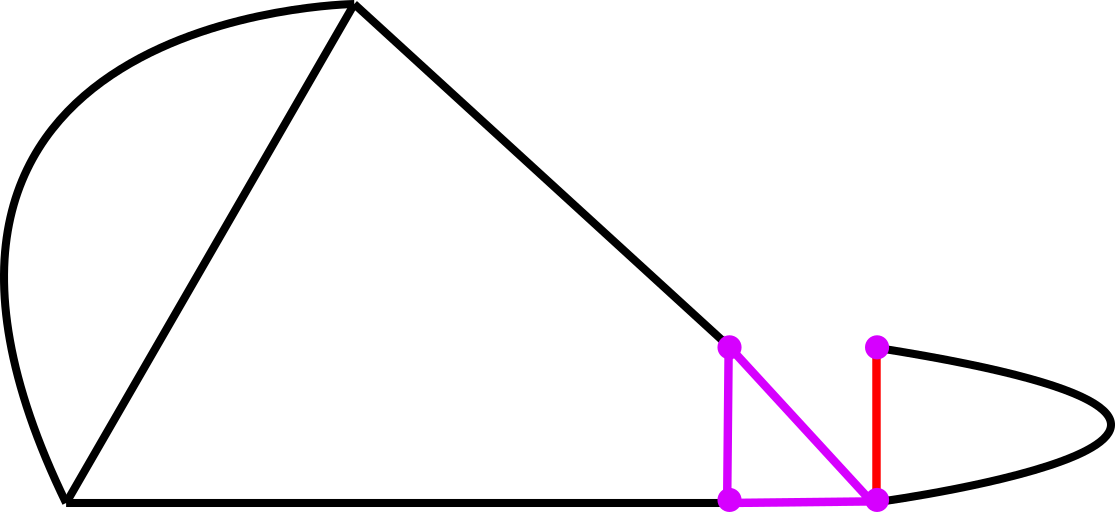}
\includegraphics[width=1.5in]{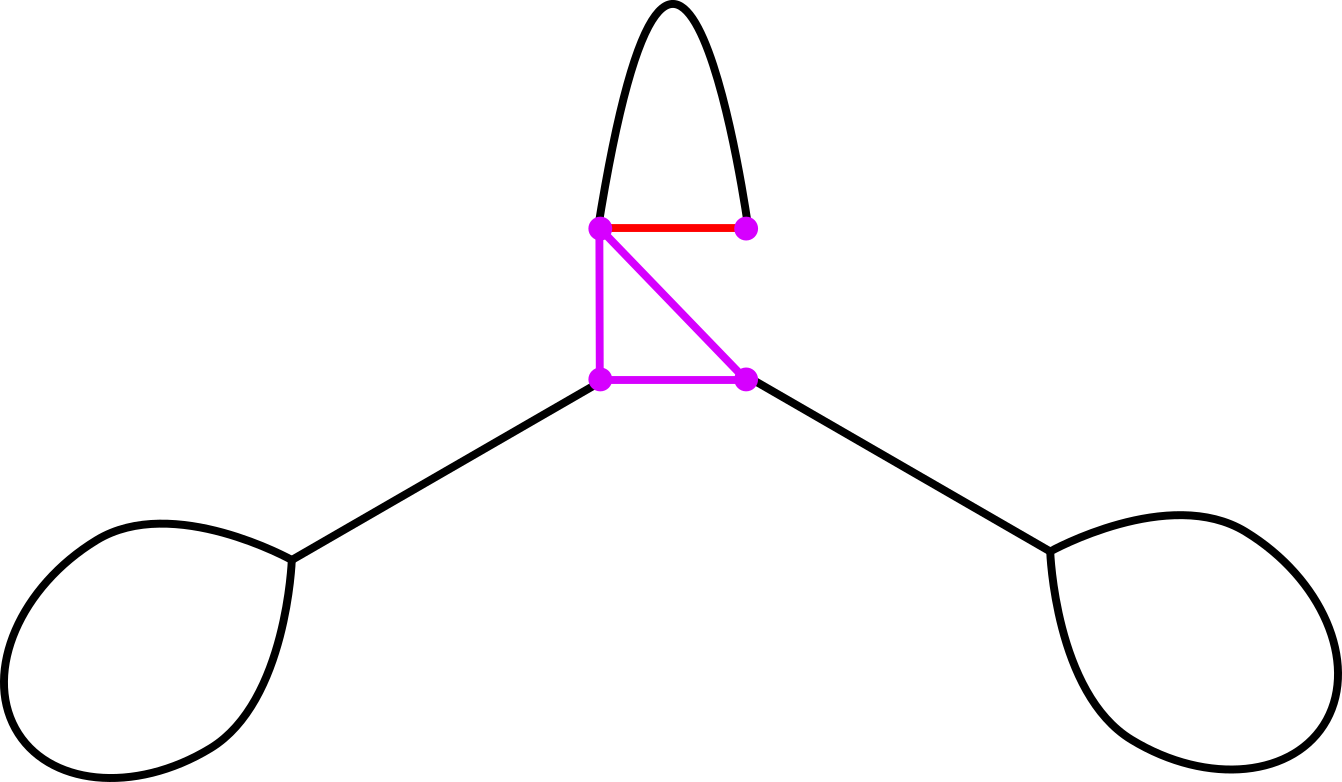}
\includegraphics[width=1.5in]{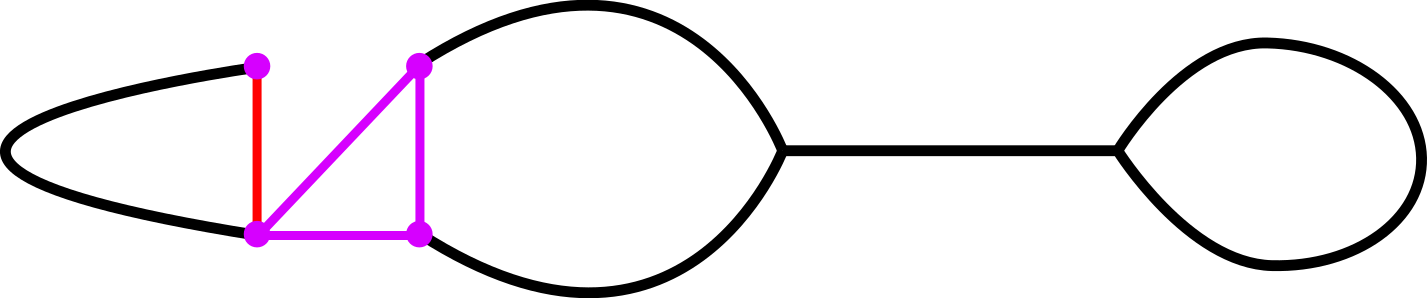}
\end{figure}

\vskip10pt

The primary maximal strongly connected component, which we will denote by $P\mA_3$, is (the interested reader can find the complete automaton written out at {\url{https://mast.queensu.ca/~cpfaff/Automaton}}.):

\vskip8pt

\begin{figure}[ht!]\label{f:automaton}
\includegraphics[width=5in]{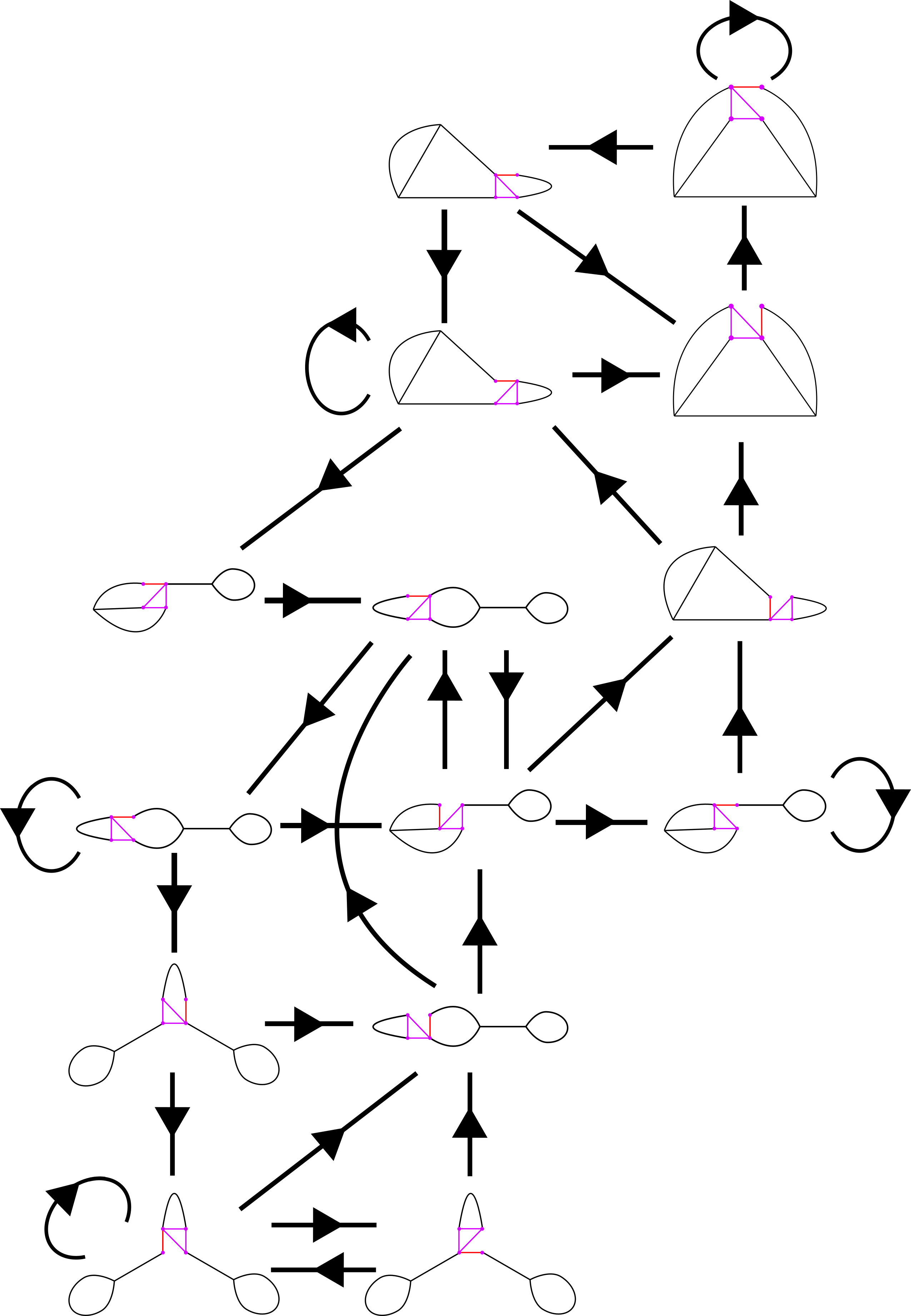}
\end{figure}

\newpage

\subsection{The map $F$}{\label{ss:MapF}}

We define as follows an edge map $F$ that is the composition of the folds determining the edges in a directed loop in $\mA_3$. Let $\Gamma = \Gamma_1, \Gamma_2, \Gamma_3, \Gamma_4, \Gamma_5$ be the graphs and $\f_1, \f_2, \f_3, \f_4$ the edge maps indicated in Figure~\ref{fig:loop}. Let $\ell_1$ be a loop in $\Gamma_1$ that takes all turns of $\Gamma_1$ except $\{a,\bar{c}\}$ and $\{a,\bar{d}\}$. We note that for each $i>1$ we have that $\f_i$ is an edge map corresponding to a permissible fold with respect to $\ell_i = \f_{i-1}(\ell_{i-1})$.

\vskip10pt

	\begin{figure}[h!]
	\includegraphics[width=6in]{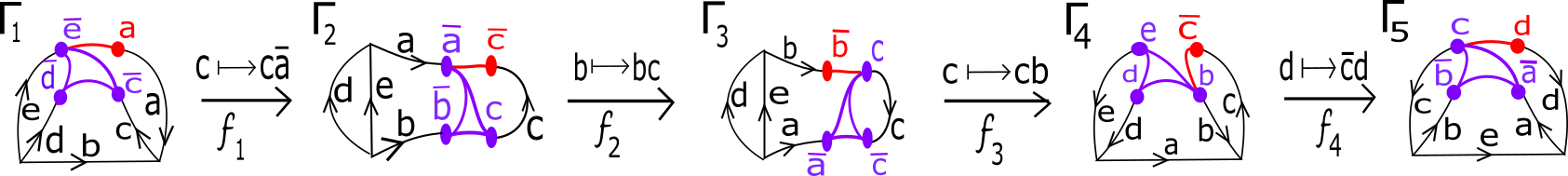}
	\captionsetup{singlelinecheck=off}
	\caption[foobar]{$\G = \G_1, \Gamma_2, \Gamma_3, \Gamma_4, \Gamma_5$ are the graphs with oriented edges $E^+(\G_i) = \{a,b,c,d,e\}$ as indicated, but with the colored inner graphs replaced by vertices. Each $\f_i \colon\G_i\to\G_{i+1}$ is an edge map sending each $\alpha\in E^+(\G_i)$ to itself except for the indicated map. Let $\ell_i$ be a comprehensive loop in $\G_i$ such that $\ell_i$ takes each turn at each valence-$3$ vertex of $\G_i$ and takes a turn $\{d_1, d_2\}$ at the valence-$4$ vertex of $\G_i$ if and only if $\{d_1, d_2\}$ is an edge in the graph depicted in the valence-$4$ vertex of $\G_i$ in the figure. The graph replacing the valence-$4$ vertex of $\G_i$ in the figure is defined by $\ell_i$ as follows. The embedded graph has a vertex for each direction at the valence-$4$ vertex of $\G_i$, represented at the beginning of the corresponding oriented edge, and there is an edge $\{d_1, d_2\}$ in the embedded graph precisely when $\{d_1, d_2\}$ is a turn taken by $\ell_i$.
	}
	\label{fig:loop}
	\end{figure}

	Let $\sigma$ be the cyclic permutation of edge labels $(a,d,b,e,\bar{c},\bar{a},\bar{d},\bar{b},\bar{e},c)$. For $4 < i \leq 40$, let $\f_i:\Gamma_i\to\Gamma_{i+1}$ be the edge map such that if $\f_{i-4}$ corresponds to the fold $\{d_1, d_2\}$, then $\f_i$ corresponds to the fold $\{\sigma(d_1), \sigma(d_2)\}$. By construction, $\f_i$ corresponds to a permissible fold with respect to $\ell_i = \f_{i-1}(\ell_{i-1})$.
	Note that $\Gamma_5$ is equal to $\Gamma$, but for a permutation on edge labels by $\sigma$. Moreover, note that the order of $\sigma$ is 10.
	
	Finally, let $F' \coloneqq \f_{40} \circ \f_{39} \circ \dots \circ \f_1:\Gamma \to \Gamma$. $F$ will denote the lowest transparent power of $F'$. The edge maps of $F'$ are:
	
	$$
a\mapsto\bar{d}e\bar{d}bc\bar{e}d\bar{c}\bar{a}\bar{e}bc\bar{d}e\bar{d}e\bar{d}ea\bar{b}eac\bar{d}e\bar{d}e\bar{d}ea\bar{b}eac\bar{d}e\bar{d}e\bar{d}ea\bar{b}ea\bar{b}eac\bar{d}e\bar{c}\bar{b}d\bar{e}d\bar{c}\bar{b}ea\bar{b}ea
\bar{b}eac\bar{d}e\bar{d}e\bar{d}bc\bar{e}d\bar{c}
\bar{a}\bar{e}bc
$$
$$
\bar{d}e\bar{d}e\bar{d}ea\bar{b}eac\bar{d}e\bar{d}e\bar{d}ea\bar{b}
eac\bar{d}e\bar{d}e\bar{d}bc\bar{d}e\bar{d}bc\bar{d}e\bar{d}bc\bar{e}d\bar{c}\bar{a}\bar{e}b
c\bar{d}e\bar{d}e\bar{d}ea\bar{b}eac\bar{d}e\bar{d}e\bar{d}bc\bar{d}e\bar{d}bc\bar{d}e\bar{d}bc\bar{e}d\bar{c}\bar{a}\bar{e}bc\bar{d}e\bar{d}
$$
$$
e\bar{d}eabar{b}eac\bar{d}e\bar{d}e\bar{d}bc\bar{d}e\bar{d}bc\bar{d}e\bar{d}bc\bar{e}d\bar{c}\bar{a}\bar{e}b\bar{a}\bar{e}b\bar{a}\bar{e}d\bar{e}d\bar{e}d\bar{c}\bar{b}eac\bar{d}e\bar{c}\bar{b}d\bar{e}d\bar{e}d\bar{c}\bar{a}\bar{e}b\bar{a}\bar{e}b\bar{a}\bar{e}bc\bar{d}e\bar{d}bc\bar{e}d\bar{c}\bar{a}\bar{e}b\bar{a}\bar{e}b\bar{a}\bar{e}
$$
$$
bc\bar{d}e\bar{d}e\bar{d}e\bar{d}bc\bar{e}d\bar{c}\bar{a}\bar{e}bc\bar{d}e\bar{d}e\bar{d}ea\bar{b}ea\bar{b}eac\bar{d}e\bar{c}\bar{b}d\bar{e}d\bar{c}\bar{b}d\bar{e}d\bar{c}\bar{b}d\bar{e}d\bar{e}d\bar{c}\bar{a}\bar{e}b\bar{a}\bar{e}d\bar{e}d\bar{e}d\bar{c}\bar{a}\bar{e}b\bar{a}\bar{e}d\bar{e}d\bar{e}d\bar{c}\bar{b}eac\bar{d}e\bar{c}\bar{b}d\bar{e}d
$$
$$
\bar{e}d\bar{c}\bar{a}\bar{e}b\bar{a}\bar{e}b\bar{a}\bar{e}bc\bar{d}e\bar{d}bc\bar{e}d\bar{c}\bar{a}\bar{e}b\bar{a}\bar{e}b\bar{a}\bar{e}bc\bar{d}e\bar{d}bc\bar{e}d\bar{c}\bar{a}\bar{e}b\bar{a}\bar{e}b\bar{a}\bar{e}bc\bar{d}e\bar{d}bc\bar{d}e\bar{d}bc\bar{e}d\bar{c}\bar{a}\bar{e}bc\bar{d}e\bar{d}e\bar{d}ea\bar{b}eac\bar{d}e\bar{d}e\bar{d}bc\bar{d}$$
$$
e\bar{d}bc\bar{d}e\bar{d}bc\bar{e}d\bar{c}\bar{a}\bar{e}b
$$

$$b\mapsto d\bar{e}d\bar{e}d\bar{c}\bar{a}\bar{e}b\bar{a}\bar{e}d\bar{e}d\bar{e}d\bar{c}\bar{a}\bar{e}b\bar{a}\bar{e}d\bar{e}d\bar{e}d\bar{c}\bar{b}eac\bar{d}e\bar{c}\bar{b}d\bar{e}d\bar{e}d\bar{c}\bar{a}\bar{e}b\bar{a}\bar{e}b\bar{a}\bar{e}bc\bar{d}e\bar{d}bc\bar{e}d\bar{c}\bar{a}\bar{e}b\bar{a}\bar{e}b\bar{a}\bar{e}d\bar{e}d\bar{e}d\bar{c}\bar{a}\bar{e}b\bar{a}\bar{e}d
$$
$$
\bar{e}d\bar{e}d\bar{c}\bar{a}\bar{e}b\bar{a}\bar{e}d\bar{e}d\bar{e}d\bar{c}\bar{b}eac\bar{d}e\bar{c}\bar{b}d\bar{e}d\bar{e}d\bar{c}\bar{a}\bar{e}b\bar{a}\bar{e}b\bar{a}\bar{e}bc\bar{d}e\bar{d}bc\bar{e}d\bar{c}\bar{a}\bar{e}\bar{e}b\bar{a}\bar{e}b\bar{a}\bar{e}d\bar{e}d\bar{e}d\bar{c}\bar{a}\bar{e}b\bar{a}\bar{e}d\bar{e}d\bar{e}d\bar{c}\bar{a}\bar{e}b\bar{a}\bar{e}d\bar{e}d\bar{e}d\bar{c}
$$
$$
\bar{b}eac\bar{d}e\bar{c}\bar{b}d\bar{e}d\bar{c}\bar{b}d\bar{e}d\bar{c}\bar{b}ea\bar{b}ea\bar{b}eac\bar{d}e\bar{d}e\bar{d}bc\bar{e}d\bar{c}\bar{a}\bar{e}bc\bar{d}e\bar{d}e\bar{d}ea\bar{b}ea\bar{b}eac\bar{d}e\bar{c}\bar{b}d\bar{e}d\bar{c}\bar{b}d\bar{e}d\bar{c}\bar{b}d\bar{e}d\bar{e}d\bar{c}\bar{a}\bar{e}b\bar{a}\bar{e}d\bar{e}d\bar{e}d\bar{c}
$$
$$
\bar{a}\bar{e}b\bar{a}\bar{e}d\bar{e}d\bar{e}
d\bar{c}\bar{b}eac\bar{d}e\bar{c}\bar{b}d\bar{e}d\bar{e}d\bar{c}\bar{a}\bar{e}b\bar{a}\bar{e}b\bar{a}\bar{e}bc\bar{d}e\bar{d}bc\bar{e}d\bar{c}\bar{a}\bar{e}b\bar{a}\bar{e}b\bar{a}\bar{e}bc\bar{d}e\bar{d}bc\bar{e}d\bar{c}\bar{a}\bar{e}b\bar{a}\bar{e}b\bar{a}\bar{e}bc\bar{d}e\bar{d}bc\bar{d}e\bar{d}bc\bar{e}d\bar{c}\bar{a}\bar{e}b
$$
$$
c\bar{d}e\bar{d}e\bar{d}ea\bar{b}eac\bar{d}e\bar{d}e\bar{d}bc\bar{d}e\bar{d}bc\bar{d}e\bar{d}bc\bar{e}d\bar{c}\bar{a}\bar{e}b\bar{a}\bar{e}b\bar{a}\bar{e}d\bar{e}d\bar{e}d\bar{c}\bar{b}eac\bar{d}e\bar{c}\bar{b}d\bar{e}d\bar{e}d\bar{c}\bar{a}\bar{e}b\bar{a}\bar{e}b\bar{a}\bar{e}bc\bar{d}e\bar{d}bc\bar{e}d\bar{c}\bar{a}\bar{e}b\bar{a}\bar{e}b\bar{a}\bar{e}
$$
$$
bc\bar{d}e\bar{d}bc\bar{e}d\bar{c}\bar{a}\bar{e}b\bar{a}\bar{e}b\bar{a}\bar{e}d\bar{e}d\bar{e}d\bar{c}\bar{a}\bar{e}b\bar{a}\bar{e}d\bar{e}d\bar{e}d\bar{c}\bar{a}\bar{e}b\bar{a}\bar{e}d\bar{e}d\bar{e}d\bar{c}\bar{b}eac\bar{d}e\bar{c}\bar{b}d\bar{e}d\bar{e}d\bar{c}\bar{a}\bar{e}b\bar{a}\bar{e}b\bar{a}\bar{e}bc\bar{d}e\bar{d}bc\bar{e}d\bar{c}\bar{a}\bar{e}b\bar{a}\bar{e}b\bar{a}\bar{e}
$$
$$
d\bar{e}d\bar{e}d\bar{c}\bar{a}\bar{e}b\bar{a}\bar{e}d\bar{e}d
\bar{e}d\bar{c}\bar{a}\bar{e}b\bar{a}\bar{e}d\bar{e}d\bar{e}d\bar{c}\bar{b}eac\bar{d}e\bar{c}\bar{b}d\bar{e}d\bar{e}d\bar{c}\bar{a}\bar{e}b\bar{a}\bar{e}b\bar{a}\bar{e}bc\bar{d}e\bar{d}bc\bar{e}d\bar{c}\bar{a}\bar{e}b\bar{a}\bar{e}b\bar{a}\bar{e}d\bar{e}d\bar{e}d\bar{c}\bar{a}\bar{e}b\bar{a}\bar{e}d\bar{e}d\bar{e}d\bar{c}
$$
$$
\bar{a}\bar{e}b\bar{a}\bar{e}d\bar{e}d\bar{e}d\bar{c}\bar{b}eac\bar{d}e\bar{c}\bar{b}d\bar{e}d\bar{c}\bar{b}d\bar{e}d\bar{c}\bar{b}ea\bar{b}ea\bar{b}eac\bar{d}e\bar{d}e\bar{d}bc\bar{e}d\bar{c}\bar{a}\bar{e}bc\bar{d}e\bar{d}e\bar{d}ea\bar{b}ea\bar{b}eac\bar{d}e\bar{c}\bar{b}d\bar{e}d\bar{c}\bar{b}d\bar{e}d\bar{c}\bar{b}d\bar{e}d\bar{e}d
$$
$$
\bar{c}\bar{a}\bar{e}b\bar{a}\bar{e}d\bar{e}d\bar{e}d\bar{c}$$

$$c\mapsto \bar{a}\bar{e}b\bar{a}\bar{e}d\bar{e}d\bar{e}d\bar{c}\bar{b}eac\bar{d}e\bar{c}\bar{b}d\bar{e}d\bar{e}d\bar{c}\bar{a}\bar{e}b\bar{a}\bar{e}b\bar{a}\bar{e}bc\bar{d}e\bar{d}bc\bar{e}d\bar{c}\bar{a}\bar{e}b\bar{a}\bar{e}b\bar{a}\bar{e}bc\bar{d}e\bar{d}bc\bar{e}d\bar{c}\bar{a}\bar{e}b\bar{a}\bar{e}b\bar{a}\bar{e}bc\bar{d}e\bar{d}bc\bar{d}e\bar{d}bc
$$
$$
\bar{e}d\bar{c}\bar{a}\bar{e}bc\bar{d}e\bar{d}e\bar{d}ea\bar{b}eac\bar{d}e\bar{d}e\bar{d}bc\bar{d}e\bar{d}bc\bar{d}e\bar{d}bc\bar{e}d\bar{c}\bar{a}\bar{e}b\bar{a}\bar{e}b\bar{a}\bar{e}d\bar{e}d\bar{e}d\bar{c}\bar{b}eac\bar{d}e\bar{c}\bar{b}d\bar{e}d\bar{e}d\bar{c}\bar{a}\bar{e}b\bar{a}\bar{e}b\bar{a}\bar{e}bc\bar{d}e\bar{d}bc\bar{e}d\bar{c}\bar{a}
$$
$$
\bar{e}b\bar{a}\bar{e}b
\bar{a}\bar{e}bc\bar{d}e\bar{d}bc\bar{e}d\bar{c}\bar{a}\bar{e}b\bar{a}\bar{e}b\bar{a}\bar{e}d\bar{e}d\bar{e}d\bar{c}\bar{a}
\bar{e}b\bar{a}\bar{e}d\bar{e}d\bar{e}d\bar{c}\bar{a}\bar{e}b\bar{a}\bar{e}d\bar{e}d\bar{e}d\bar{c}\bar{b}eac\bar{d}e\bar{c}\bar{b}d\bar{e}d\bar{e}d\bar{c}\bar{a}\bar{e}b\bar{a}\bar{e}b\bar{a}\bar{e}bc\bar{d}e\bar{d}bc\bar{e}d\bar{c}
$$
$$
\bar{a}\bar{e}b\bar{a}\bar{e}b\bar{a}\bar{e}d\bar{e}d\bar{e}d\bar{c}\bar{a}\bar{e}b\bar{a}\bar{e}d\bar{e}d\bar{e}d\bar{c}\bar{a}\bar{e}b\bar{a}\bar{e}d\bar{e}d\bar{e}d\bar{c}\bar{b}eac\bar{d}e\bar{c}\bar{b}d\bar{e}d\bar{e}d\bar{c}\bar{a}\bar{e}b\bar{a}\bar{e}b\bar{a}\bar{e}bc\bar{d}e\bar{d}bc\bar{e}d\bar{c}\bar{a}\bar{e}b\bar{a}\bar{e}b\bar{a}\bar{e}d\bar{e}d\bar{e}d\bar{c}\bar{a}\bar{e}
$$
$$
b\bar{a}\bar{e}d\bar{e}d\bar{e}d\bar{c}\bar{a}\bar{e}b\bar{a}\bar{e}d\bar{e}d\bar{e}d\bar{c}\bar{b}eac\bar{d}e\bar{c}\bar{b}d\bar{e}d\bar{c}\bar{b}d\bar{e}d\bar{c}\bar{b}ea\bar{b}ea\bar{b}eac\bar{d}e\bar{d}e\bar{d}bc\bar{e}d\bar{c}\bar{a}\bar{e}bc\bar{d}e\bar{d}e\bar{d}ea\bar{b}ea\bar{b}eac\bar{d}e\bar{c}\bar{b}d\bar{e}d\bar{c}\bar{b}
$$
$$
d\bar{e}d\bar{c}\bar{b}d\bar{e}d\bar{e}d\bar{c}\bar{a}\bar{e}b\bar{a}\bar{e}d\bar{e}d\bar{e}d\bar{c}\bar{b}eac\bar{d}e\bar{c}\bar{b}d\bar{e}d\bar{c}\bar{b}d\bar{e}d\bar{c}\bar{b}d\bar{e}d\bar{e}d\bar{c}\bar{a}\bar{e}b\bar{a}\bar{e}d\bar{e}d\bar{e}d\bar{c}\bar{b}eac\bar{d}e\bar{c}\bar{b}d\bar{e}d\bar{c}\bar{b}d\bar{e}d\bar{c}\bar{b}ea\bar{b}ea\bar{b}eac
$$	
$$
\bar{d}e\bar{c}\bar{b}d\bar{e}d\bar{c}\bar{b}ea\bar{b}ea\bar{b}eac\bar{d}e\bar{c}\bar{b}d\bar{e}d\bar{c}\bar{b}ea\bar{b}ea\bar{b}eac\bar{d}e\bar{d}e\bar{d}bc\bar{e}d\bar{c}\bar{a}\bar{e}bc\bar{d}e\bar{d}e\bar{d}ea\bar{b}eac\bar{d}e\bar{d}e\bar{d}ea\bar{b}eac\bar{d}e\bar{d}e\bar{d}bc\bar{d}e\bar{d}bc\bar{d}
$$
$$
e\bar{d}bbc\bar{e}d\bar{c}\bar{a}\bar{e}b\bar{a}\bar{e}b\bar{a}\bar{e}d\bar{e}d\bar{e}d\bar{c}\bar{b}eac\bar{d}e\bar{c}\bar{b}d\bar{e}d\bar{e}d\bar{c}\bar{a}\bar{e}b\bar{a}\bar{e}b\bar{a}\bar{e}bc\bar{d}e\bar{d}bc\bar{d}e\bar{d}bc\bar{e}d\bar{c}\bar{a}\bar{e}bc\bar{d}e\bar{d}e\bar{d}e	a\bar{b}eac\bar{d}e\bar{d}e\bar{d}ea\bar{b}eac
$$
$$
\bar{d}e\bar{d}e\bar{d}ea\bar{b}ea\bar{b}eac\bar{d}e\bar{c}\bar{b}d\bar{e}d\bar{c}\bar{b}ea\bar{b}ea\bar{b}eac\bar{d}e\bar{c}\bar{b}d\bar{e}d\bar{c}\bar{b}ea\bar{b}ea\bar{b}eac\bar{d}e\bar{d}e\bar{d}bc\bar{e}d\bar{c}\bar{a}\bar{e}bc\bar{d}e\bar{d}e\bar{d}ea\bar{b}ea\bar{b}eac\bar{d}e\bar{c}\bar{b}d\bar{e}d
$$
$$
\bar{c}\bar{b}d\bar{e}d\bar{c}\bar{b}d\bar{e}d\bar{e}d\bar{c}\bar{a}\bar{e}b\bar{a}\bar{e}d\bar{e}d\bar{e}d\bar{c}\bar{b}eac\bar{d}e\bar{c}\bar{b}d\bar{e}d\bar{c}\bar{b}d\bar{e}d\bar{c}\bar{b}d\bar{e}d\bar{e}d\bar{c}\bar{a}\bar{e}b\bar{a}\bar{e}d\bar{e}d\bar{e}d\bar{c}\bar{b}eac\bar{d}e\bar{c}\bar{b}d\bar{e}d\bar{c}\bar{b}d\bar{e}d\bar{c}\bar{b}d\bar{e}d\bar{e}d\bar{c}\bar{a}\bar{e}
$$
$$
b\bar{a}\bar{e}d\bar{e}d\bar{e}d\bar{c}\bar{a}\bar{e}b
\bar{a}\bar{e}d\bar{e}d\bar{e}d\bar{c}\bar{b}eac\bar{d}e\bar{c}\bar{b}d\bar{e}d\bar{e}d\bar{c}\bar{a}\bar{e}b\bar{a}\bar{e}b\bar{a}\bar{e}bc\bar{d}e\bar{d}bc\bar{e}d\bar{c}\bar{a}\bar{e}b\bar{a}\bar{e}b\bar{a}\bar{e}d\bar{e}d\bar{e}d\bar{c}\bar{a}\bar{e}b\bar{a}\bar{e}d\bar{e}d\bar{e}d\bar{c}\bar{a}\bar{e}b\bar{a}\bar{e}d\bar{e}
$$
$$
d\bar{e}d\bar{c}\bar{b}eac\bar{d}e\bar{c}\bar{b}d\bar{e}d$$

	$$d\mapsto ea\bar{b}ea\bar{b}eac\bar{d}e\bar{c}\bar{b}d\bar{e}d\bar{c}\bar{b}ea\bar{b}ea\bar{b}eac\bar{d}e\bar{c}\bar{b}d\bar{e}d\bar{c}\bar{b}ea\bar{b}ea\bar{b}eac\bar{d}e\bar{d}e\bar{d}bc\bar{e}d\bar{c}\bar{a}\bar{e}bc\bar{d}e\bar{d}e\bar{d}ea\bar{b}ea\bar{b}eac\bar{d}e\bar{c}\bar{b}d\bar{e}d\bar{c}
$$
$$
\bar{b}d\bar{e}d\bar{c}\bar{b}d\bar{e}d\bar{e}d\bar{c}\bar{a}\bar{e}b\bar{a}\bar{e}d\bar{e}d\bar{e}d\bar{c}\bar{b}eac\bar{d}e\bar{c}\bar{b}d\bar{e}d\bar{c}\bar{b}d\bar{e}d\bar{c}\bar{b}ea\bar{b}ea\bar{b}eac\bar{d}	e\bar{c}\bar{b}d\bar{e}d\bar{c}\bar{b}ea\bar{b}ea\bar{b}eac\bar{d}e\bar{c}\bar{b}d\bar{e}d\bar{c}\bar{b}ea\bar{b}ea\bar{b}ea
$$
$$
c\bar{d}e\bar{d}e\bar{d}bc\bar{e}d\bar{c}\bar{a}\bar{e}bc\bar{d}e\bar{d}e\bar{d}ea\bar{b}ea\bar{b}eac\bar{d}e\bar{c}\bar{b}d\bar{e}d\bar{c}\bar{b}d\bar{e}d\bar{c}\bar{b}d\bar{e}d\bar{e}d\bar{c}\bar{a}\bar{e}b\bar{a}\bar{e}d\bar{e}d\bar{e}d\bar{c}\bar{b}eac\bar{d}e\bar{c}\bar{b}d\bar{e}d\bar{c}\bar{b}d\bar{e}d\bar{c}\bar{b}ea\bar{b}ea\bar{b}e
$$
$$
ac\bar{d}e\bar{c}\bar{b}d\bar{e}d\bar{c}\bar{b}ea
\bar{b}ea\bar{b}eac\bar{d}e\bar{c}\bar{b}d\bar{e}d\bar{c}\bar{b}ea\bar{b}ea\bar{b}eac\bar{d}e\bar{d}e\bar{d}bc\bar{e}d\bar{c}\bar{a}\bar{e}bc\bar{d}e\bar{d}e\bar{d}ea\bar{b}eac\bar{d}e\bar{d}e\bar{d}ea\bar{b}eac\bar{d}e\bar{d}e\bar{d}bc\bar{d}e\bar{d}b
$$
$$
c\bar{d}e\bar{d}bc\bar{e}d\bar{c}\bar{a}\bar{e}b\bar{a}\bar{e}b\bar{a}\bar{e}d\bar{e}d\bar{e}d\bar{c}\bar{b}eac\bar{d}e\bar{c}\bar{b}d\bar{e}d\bar{e}d\bar{c}\bar{a}\bar{e}b\bar{a}\bar{e}b\bar{a}\bar{e}bc\bar{d}e\bar{d}bc\bar{d}e\bar{d}bc\bar{e}d\bar{c}\bar{a}\bar{e}bc\bar{d}e\bar{d}e\bar{d}ea\bar{b}eac\bar{d}e\bar{d}e\bar{d}ea\bar{b}ea
$$
$$
c\bar{d}e\bar{d}e\bar{d}ea\bar{b}ea\bar{b}eac\bar{d}e\bar{c}\bar{b}d\bar{e}d\bar{c}\bar{b}ea\bar{b}ea\bar{b}eac\bar{d}e
$$
	
$$e\mapsto bc\bar{d}e\bar{d}bc\bar{d}e\bar{d}bc\bar{e}d\bar{c}\bar{a}\bar{e}bc\bar{d}e\bar{d}e\bar{d}ea\bar{b}eac\bar{d}e\bar{d}e\bar{d}bc\bar{d}e\bar{d}bc\bar{d}e\bar{d}bc\bar{e}d\bar{c}\bar{a}\bar{e}bc\bar{d}e\bar{d}e\bar{d}ea\bar{b}eac\bar{d}e\bar{d}e\bar{d}bc\bar{d}e\bar{d}bc\bar{d}e\bar{d}b
$$
$$
c\bar{e}d\bar{c}\bar{a}\bar{e}b\bar{a}\bar{e}b\bar{a}\bar{e}d\bar{e}d\bar{e}d\bar{c}\bar{b}eac\bar{d}e\bar{c}\bar{b}d\bar{e}d\bar{e}d\bar{c}\bar{a}\bar{e}b\bar{a}\bar{e}b\bar{a}\bar{e}bc\bar{d}e\bar{d}bc\bar{d}e\bar{d}bc\bar{e}d\bar{c}\bar{a}\bar{e}bc\bar{d}e\bar{d}e\bar{d}ea\bar{b}eac\bar{d}e\bar{d}e\bar{d}ea\bar{b}eac\bar{d}e\bar{d}e
$$	
$$
\bar{d}ea\bar{b}ea\bar{b}eac\bar{d}e\bar{c}\bar{b}d\bar{e}d\bar{c}\bar{b}ea\bar{b}ea\bar{b}eac\bar{d}e\bar{d}e\bar{d}bc\bar{e}d\bar{c}\bar{a}\bar{e}bc\bar{d}e\bar{d}e\bar{d}ea\bar{b}eac\bar{d}e\bar{d}e\bar{d}ea\bar{b}eac\bar{d}e\bar{d}e\bar{d}bc\bar{d}e\bar{d}bc\bar{d}e\bar{d}bc\bar{e}d\bar{c}\bar{a}
$$	
$$
\bar{e}bc\bar{d}e\bar{d}e\bar{d}ea\bar{b}eac\bar{d}e\bar{d}e\bar{d}bc\bar{d}e\bar{d}bc\bar{d}e\bar{d}bc\bar{e}d\bar{c}\bar{a}\bar{e}bc\bar{d}e\bar{d}e\bar{d}ea\bar{b}eac\bar{d}e\bar{d}e\bar{d}bc\bar{d}e\bar{d}bc\bar{d}e\bar{d}bc\bar{e}d\bar{c}\bar{a}\bar{e}b\bar{a}\bar{e}b\bar{a}\bar{e}d\bar{e}d\bar{e}d\bar{c}\bar{b}e
$$
$$
ac\bar{d}e\bar{c}\bar{b}d\bar{e}d\bar{e}d\bar{c}\bar{a}\bar{e}b\bar{a}\bar{e}b\bar{a}\bar{e}bc\bar{d}e\bar{d}bc
$$		

\subsection{The turns taken by $F$}{\label{ss:Fturnstaken}}

\vskip 5pt

	\begin{df}[$\mathcal{T}_f$]
		Let $f:\G\to \G$ be an edge map.  Then define $\mathcal{T}_f$ to be the set of all turns taken by $f^k(E)$ before tightening, as $E$ ranges over $E(\G)$ and as $k$ ranges over $\ZZ_{>0}$. If $v$ is a vertex of $\G$, then $\mathcal{T}_f|_v$ is the restriction of $\mathcal{T}_f$ to turns at $v$.\\
	\end{df}

We now verify the validity for the algorithm used in \cite{g20b} to compute $\mathcal{T}_F$.

\begin{lem}{\label{l:lhwalg}}
	Let $\Gamma$ be a high-valence graph, $g:\Gamma \to \Gamma$ an edge map,  $v\in V(\Gamma)$, and $R_0$ the set of interior turns of $g$ in $\Gamma$.
	
	Given $R_0, R_1, \dots, R_i$, define $R_{i+1}$ as follows:
	
	\indent\indent Define $N_i \coloneqq \{ \{dg(t_1), dg(t_2)\} \mid \{t_1, t_2\} \in R_i\}$ and

\indent\indent\indent $R_{i+1} \coloneqq \{ \{t_1, t_2\} \mid \{t_1, t_2\} \in N_i, \{t_1, t_2\} \notin R_j \text{ for } j < i + 1\}$.
	
	Let $l$ be the first integer such that $R_l = \varnothing$ and set $R = \cup_{i=0}^{l-1}R_i$.
	Then $R = \mathcal{T}_g$.
\end{lem}
	
\begin{proof}	
	We first show $\mathcal{T}_g \subseteq R$. Let $\{d_1, d_2\} \in \mathcal{T}_g$. Then, by definition, there exists an $E \in E(\Gamma)$ and $k\in \ZZ_{>0}$ such that $g^k(E)$ takes $\{d_1, d_2\}$ before tightening.
	Thus, it suffices to show for each $E \in E(\G)$ and $k \in \ZZ_{>0}$, $R$ contains all turns taken by $g^k(E)$ before tightening. To prove this, we fix an arbitrary $E\in E(\G)$ and proceed by induction on $k$.
	
	{\bf{Base Case:}} We want to show that all turns taken by $g(E)$ before tightening are contained in $R$. By definition, all turns taken by $g(E)$ before tightening are contained in $R_0 \subseteq R$.
	
	{\bf{Inductive Step:}} Fix $j \in \ZZ_{>0}$. Suppose all turns taken by $g^j(E)$ before tightening are contained in $R$.
	
	{\bf{Inductive Step:}} We show that all turns taken by $g^{j+1}(E)$ before tightening are contained in $R$. Let $\{t_{1, j+1}, t_{2, j+1}\}$ be a turn taken by $g^{j+1}(E)$ before tightening. By Lemma~\ref{l:dirimg}, $\{t_{1, j+1}, t_{2, j+1}\}$ is either an interior turn or an exterior turn of $g$ in $g^j(E)$.
	
	If $\{t_{1, j+1}, t_{2, j+1}\}$ is an interior turn, then $\{t_{1, j+1}, t_{2, j+1}\} \in R_0 \subseteq R$.
	So we assume $\{t_{1, j+1}, t_{2, j+1}\}$ is an exterior turn. Then, there exists a turn $\{t_{1,j}, t_{2,j}\}$ taken by $g^j(E)$ before tightening such that $\{t_{1, j+1}, t_{2, j+1}\} = \{Dg(t_{1,j}), Dg(t_{2,j})\}$.
	
	By the inductive hypothesis, $\{t_{1,j}, t_{2,j}\} \in R$. Thus, there exists some $\alpha < l$ such that $\{t_{1,j}, t_{2,j}\} \in R_\alpha$.
	Then, by the definition of $R_\alpha$, $\{t_{1, j+1}, t_{2, j+1}\} = \{Dg(t_{1,j}), Dg(t_{2,j})\} \in R_\beta$ for some $\beta \leq \alpha + 1$.
	Then, since $\alpha < l$, we have $\alpha + 1 \leq l$. This implies $\beta \leq l$. But since $R_l = \varnothing$ and $\{t_{1, j+1}, t_{2, j+1}\} \in R_\beta$, we have $\beta \neq l$. Thus, $\beta < l$, and $\{t_{1, j+1}, t_{2, j+1}\} \in R_\beta \subseteq R$.
	Therefore, $\{t_{1, j+1}, t_{2, j+1}\} \in R$.
	
	This proves the claim and thus that $\mathcal{T}_g \subseteq R$.
	
	To prove $R \subseteq \mathcal{T}_g$, we let $\{d_1, d_2\} \in R$ and show $\{d_1, d_2\} \in \mathcal{T}_g$. Either $\{d_1, d_2\} \in R_0$ or $\{d_1, d_2\}\in R_k$ for some $k > 0$. We prove these cases separately:
	
	{\bf{Case 1 ($\{d_1, d_2\} \in R_0$):}} If $\{d_1, d_2\} \in R_0$, then, by definition, there exists some $E\in E(\Gamma)$ such that $\{d_1, d_2\}$ is taken by $g(E)$ before tightening. Therefore, $\{d_1, d_2\} \in \mathcal{T}_g$. This proves the claim in Case 1.
	
	{\bf{Case 2 ($\{d_1, d_2\}\in R_k$ for some $k > 0$):}} By construction, for all $i > 0$ and turns $\{a,b\} \in R_i$, there exists a turn $\{c,d\} \in R_{i-1}$ such that $\{a,b\} = \{Dg(c), Dg(d)\}$. Then, $\{d_1, d_2\}$ can be obtained by applying $k$ times $Dg$ to a turn in $R_0$, i.e. there must exist some sequence of turns $\{d_{0,1}, d_{0,2}\},\{d_{1,1}, d_{1,2}\},\dots,\{d_{k,1}, d_{k,2}\}$ such that:
	\begin{itemize}
		\item $\{d_{k,1}, d_{k,2}\}=\{d_{1}, d_{2}\}$,
		\item $\{d_{i,1}, d_{i, 2}\} \in R_i$,
		\item $\{d_{i+1,1}, d_{i+1,2}\} = \{Dg(d_{i,1}), Dg(d_{i,2})\}$ for $i = 1, 2, \dots, k-1$, and
		\item $\{d_{0,1}, d_{0,2}\} \in R_0$.
	\end{itemize}
	Thus, $\{d_1, d_2\} = \{Dg^{k}(d_{0,1}), Dg^{k}(d_{0,2})\}$. Since $\{d_{0,1}, d_{0,2}\}\in R_0$, there exists some $E\in E(\Gamma)$ satisfying that $\{d_{0,1}, d_{0,2}\}$ is taken by $g(E)$ before tightening. Therefore, $\{d_1, d_2\}$ is a turn taken by $g^{k+1}(E)$ before tightening.
	
	Thus, $\{d_1, d_2\} \in \mathcal{T}_g$, which proves the claim in Case 2.
	
	This proves that $R\subseteq \mathcal{T}_g$.
	Since $\mathcal{T}_g\subseteq R$ and $\mathcal{T}_g \supseteq R$, we must have $\mathcal{T}_g = R$, which completes the proof.\\
\end{proof}

\begin{lem}
	Suppose $f = h^n$ for some $n\in\mathbb{N}$. Then $\mT_{f} = \mT_h$
\end{lem}

\begin{proof}
	First we show $\mT_{f} \subseteq \mT_h$. Let $\{d_1, d_2\} \in \mT_{f}$. We aim to show $\{d_1, d_2\} \in \mT_{h}$. Since $\{d_1, d_2\} \in \mT_{f}$, there exists some $k\geq 1$ and $e\in E^+(\G)$ such that $f^k(e)$ takes $\{d_1, d_2\}$.
	Let $nj > k$. Since $f$ maps each $\alpha\in E^+(\G)$ over each $\beta\in E^+(\G)$ at least once, $f^i$ maps each $\alpha\in E^+(\G)$ over each $\beta\in E^+(\G)$ at least once for every $i \geq 1$.
	In particular, $f^{nj-k}(e)$ contains $e$. Thus, $h^j(e) = f^{nj}(e) = (f^k\circ f^{nj-k})(e)$ contains $(f)^k(e)$. Thus $h^j(e)$ takes $\{d_1, d_2\}$. Thus,  $\{d_1, d_2\}\in \mT_{h}$, and so $\mT_{f} \subseteq \mT_h$.
	
	Now we show $\mT_{h} \subseteq \mT_{f}$. Let $\{d_1, d_2\} \in \mT_{h}$. We aim to show $\{d_1, d_2\} \in \mT_{f}$. Since $\{d_1, d_2\} \in \mT_{h}$, there exists $k\geq 1$ and $e\in E^+(\G)$ such that $h^k(e)$ takes $\{d_1, d_2\}$. Since $h^k = f^{nk}$, then $f^{nk}(e)$ also takes $\{d_1, d_2\}$. Thus, $\{d_1, d_2\}\in\mT_{f}$, and so $\mT_{h} \subseteq \mT_{f}$.
	
Since we have shown containment in both directions, $\mT_{f} = \mT_h$.\\
\end{proof}

Using the program \cite{g20b}, one obtains:
	$$\mathcal{T}_F = \{\{d,e\},\{\bar{d},\bar{e}\},\{b,d\},\{\bar{b},c\},\{\bar{c},\bar{e}\},\{\bar{c},\bar{d}\},\{\bar{a},c\},
\{\bar{e},a\},\{b,e\},\{\bar{a},\bar{b}\}\}$$.
	
\vskip 5pt

\section{Train track maps \& (principal) fully irreducible outer automorphisms}{\label{s:tt}}

\subsection{Train track maps \& fully irreducible outer automorphisms}

	\begin{df}[Train track maps, transition matrix, Perron-Frobenius matrix]
	Let $g:\G\to\G$ an edge map. We call $g$ a \textit{train track map} if for each $k\in\ZZ_{>0}$ and $e\in E(\G)$, we have that $g^k(e)$ does not contain $e'\overline{e'}$ for any $e'\in E(\G)$. We call the train track map $g$ \textit{expanding} if for each edge $e \in E(\Gamma)$ we have that $|g^n(e)|\to\infty$ as $n\to\infty$, where for a path $\gamma$ we use $|\gamma|$ to denote the number of edges $\gamma$ traverses (with multiplicity). Note that, apart from our not requiring a ``marking,'' these definitions coincide with those in  \cite{bh92} when $g$ is instead viewed as a homotopy equivalence of graphs. We remind the reader that we often drop discussion of the marking when it does not affect our arguments.
	
	The \textit{transition matrix} of a train track map $g:\Gamma \to \Gamma$ is the square $|E^+(\Gamma)| \times |E^+(\Gamma)|$ matrix $[a_{ij}]$ such that $a_{ij}$, for each $i$ and $j$, is the number of times $g(e_i)$ contains either $e_j$ or $\overline{e_j}$.
	A transition matrix $A=[a_{ij}]$ is \textit{Perron-Frobenius (PF)} if there exists an $N$ such that, for each $k \geq N$, we have that $A^k$ is strictly positive.
	
	Since we do not explicitly use the definition of a Nielsen path, periodic Nielsen path (PNP), or indivisible Nielsen path (iNP), we refer the reader to \cite{bh92} for the definitions.	
	\end{df}

\begin{df}[Irreducible, fully irreducible] We call a train track map \emph{irreducible} if it has no proper invariant subgraph with a noncontractible component. It is useful to note that an edge map whose transition matrix is Perron-Frobenius must be irreducible. An outer automorphism $\vphi\in\out$ is \emph{fully irreducible} if no positive power preserves the conjugacy class of a proper free factor of $F_r$. Bestvina and Handel \cite{bh92} proved that every fully irreducible outer automorphism admits expanding irreducible train track representatives.\\
\end{df}

\subsection{Whitehead graphs}{\label{ss:WGs}}
Local Whitehead graphs, local stable Whitehead graphs, and ideal Whitehead graphs were introduced in \cite{hm11} and give information about how images of edges pass through vertices.
We give definitions here only in the circumstance of no PNPs, as this will always be the case for us. We further assume that we have replaced any edge map $g$ by a transparent power (in the sense of Definition~\ref{d:Transparent}).

	Let $g:\G \to \G$ be a train track map. The \textit{local Whitehead graph} $LW(g; v)$ at a $v \in V(\Gamma)$ has a vertex for each direction at $v$ and an edge connecting the vertices corresponding to a pair of directions $\{d_1,d_2\}$ at $v$ precisely when the turn $\{d_1,d_2\}$ is $g$-taken. Given a fixed vertex $v$, we obtain the \emph{local stable Whitehead graph} $SW(g;v)$ from $LW(g;v)$ by restricting to the fixed direction vertices and the edges between them.
In this situation of no PNPs, if $g$ represents a fully irreducible outer automorphism $\vphi$, then the \emph{ideal Whitehead graph} $IW(\vphi)$ for $\vphi$ is isomorphic to the disjoint union $\bigsqcup SW(g;v)$ taken over all fixed vertices of $g$. Justification of the ideal Whitehead graph being an outer automorphism and conjugacy class invariant can be found in \cite{hm11, Thesis}.

\begin{rmk}
	For a transparent train track map  $g:\G \to \G$, the set of edges in $LW(g;v)$ is equal to $\mathcal{T}_g|_v$.\\
\end{rmk}

\subsection{Full irreducibility criterion}{\label{ss:fic}}

We will use the following criterion for proving that a train track map represents a fully irreducible outer automorphism.

\begin{prop}\cite[Proposition 4.1]{IWGII}(\emph{Full Irreducibility Criterion (FIC)})\label{prop:FIC}
	Let $g\colon \Gamma \to \Gamma$ be a PNP-free, irreducible train track representative of $\vphi \in Out(F_r)$. Suppose that the transition matrix for $g$ is Perron-Frobenius and that all the local Whitehead graphs are connected. Then $\vphi$ is a fully irreducible outer automorphism.\\
\end{prop}

\subsection{Principal fully irreducible outer automorphisms}

Principal fully irreducible outer automorphisms (sometimes just called principal outer automorphisms) were introduced in \cite{stablestrata} as analogues of pseudo-Anosov surface homeomorphisms with only $3$-pronged singularities.

\begin{df}[principal fully irreducible outer automorphism]
As in \cite{stablestrata}, we call a fully irreducible outer automorphism $\vphi \in Out(F_r)$ \emph{principal} if $IW(\vphi)$ is the disjoint union of $2r-3$ triangles.
\end{df} 
By the main theorem of \cite{loneaxes}, principal outer automorphisms are lone axis outer automorphisms. Relevant for us will be that each of their train track representatives has a unique Stallings fold decomposition, the Stallings fold decompositions of any two of these train track representatives are cyclic permutations of each other, and none of these train track representatives has a PNP.

\section{Main Proof}{\label{s:MainProof}}

The following lemma will tell us that loops in $\mA_3$ yield train track maps.

\vskip1pt

	\begin{lem}
	Let $\G$ be a high-valence graph, $\ell$ a tight comprehensive loop on $\G$, and $\mF = (\f_1, \f_2, \dots, \f_n)$ a sequence of permissible folds in $\G$ such that the following are satisfied by $f=:\f_n\circ \cdots\circ \f_1$:

\noindent 1. There exists an ornamentation-preserving (but for the addition of primes) isomorphism between $f(\G)$ and $\G$ and
	
\noindent 2. the isomorphism induces an identification of the set of turns taken by $\ell$ with the set of turns taken by $f(\ell)$.
	
Then, $f$ is a train track map.
\end{lem}

\begin{proof}
	Let $k\in\ZZ_{>0}$ and $e\in E(\G)$ be fixed and arbitrary. It suffices to show that $f^k(e)$ is tight.
	By (1) and (2), $f$ is also a sequence of permissible folds on $f(\G)$. Inductively, $f$ is a sequence of permissible folds on $f^j(\G)$ for any $j$. Therefore, $f^j$ is a permissible sequence of folds on $f(\G)$ for any $j$.
	In particular, $f^k$ is a permissible sequence of folds on $\G$. Then, by Lemma~\ref{l:F local inj}, $f^k(e)$ is tight, as required.\\
\end{proof}

\begin{prop}
	The map $F$ of \S \ref{ss:MapF} represents a principal fully irreducible outer automorphism.
\end{prop}
	
	\begin{proof}
In light of Proposition \ref{prop:FIC} (and the fact that a Perron-Frobenius transition matrix implies irreducibility for a transparent map), for full irreducibility, it suffices to show that $F$ is an expanding, Perron-Frobenius train track map with no PNPs, and that the local Whitehead graphs connected. Recall that $F$ has been taken to be transparent.

We first prove $F$ is an expanding irreducible train track map with Perron-Frobenius transition matrix. There does not exist a direction $d$ in $\Gamma$ such that $\{d,d\}$ is an element of $\mathcal{T}_F$. That is, for each $e\in E(\Gamma)$ and $k\in \ZZ_{>0}$, we have that $F^k(e)$ does not contain $e'\overline{e'}$ for any $e'\in E(\G)$. Thus, $F$ is a train track map. Since $F$ is a train track map, so that there is no cancellation in any edge $F^k$-images, and $|F(\alpha)| > 1$ for each $\alpha\in E(\G)$, it follows that $|F^k(\alpha)|\to \infty$ as $k\to\infty$ for each $\alpha\in E(\G)$. By inspection, $F$ maps each $\alpha\in E^+(\G)$ over each $\beta\in E^+(\G)$ at least once. Thus, each entry in the transition matrix, $A$, of $F$ is strictly positive. It follows that $A^k$ is strictly positive for each $k\in\ZZ_{>0}$. Thus $F$ is Perron-Frobenius.

One can show $F$ has no PNPs using the sage train track package of Coulbois (\cite{c12}) or the methodology of \cite{IWGII}, \cite[Example 3.4]{hm11}, \cite{p15}, \cite[Lemma 4.7]{counting}.

We now show that for each $v\in V(\G)$, we have that $LW(v;F)$ is connected. Let $v_1$ be the vertex of $\G$ with directions $\{\overline{d}, \overline{e}, \overline{c}, a\}$, and $v_2$ the vertex of $\G$ with directions $\{d, e, b\}$, and $v_3$ be the vertex of $\G$ with directions $\{c, \overline{a}, \overline{b}\}$, and $D_{v_i}$ the set of directions at $v_i$. Restricting $\mathcal{T}_F$ to each vertex, one obtains:
	
	$\mathcal{T}_F|_{v_1} = \{\{\bar{d},\bar{e}\}, \{\bar{c},\bar{e}\}, \{\bar{c},\bar{d}\}, \{\bar{e},a\}\}$
	
	$\mathcal{T}_F|_{v_2} = \{\{d,e\}, \{b,d\}, \{b,e\}\}$
	
	$\mathcal{T}_F|_{v_3} = \{\{\bar{b},c\}, \{\bar{a},c\}, \{\bar{a},\bar{b}\}\}$.
	
	Thus, $LW(v_1;F) = (D_{v_1}, \mathcal{T}_F|_{v_1}), LW(v_2;F) = (D_{v_2}, \mathcal{T}_F|_{v_2})$, and $LW(v_3;F) = (D_{v_3}, \mathcal{T}_F|_{v_3})$, which are connected graphs, as desired. And $F$ is a fully irreducible outer automorphism. Further, since the only nonperiodic direction is $a$, this implies that the ideal Whitehead graph is a union of 3 triangles, and thus $\vphi$ is a principal fully irreducibly outer automorphism.\\
\end{proof}

\begin{prop}\label{p:AllLoopsPrincipal}
There exists a power $F^p$ of $F$ satisfying the following. Suppose $L$ is a loop in the automata containing the loop $L'$ giving the fold sequence for $F^p$. Let $g$ denote the train track map obtained by the sequence of folds the directed edges in $L$ represent. Then $g$ represents a principal fully irreducible outer automorphism.
\end{prop}

\begin{proof}
Note that we have already taken $F$ to be transparent.

We first need that $g$ contains no PNPs. This argument is in the proof of \cite[Lemma 5.3]{stablestrata} only mildly disguised. The first observation is that the condition in \cite[Lemma 5.3]{stablestrata} that a periodic fold line satisfies $\mL\in B^k(A_\vphi,R,\veps)$ is simply to ensure that a Stallings fold composition for $g$ contains the Stallings fold decomposition for $F^p$, which holds true in our case. The condition of having a ``witness loop'' is replaced by the comprehensive loop (and its images) determining our automaton. The rest of the proof is then exactly the same.

By the proof of \cite[Lemma 5.3]{stablestrata}, we further have that $F^p$ and $g$ have the same local Whitehead graphs and stable Whitehead graphs and that the transition matrix for $g$ is Perron-Frobenius. Thus, since the local Whitehead graphs for $F$, hence $F^p$, were connected, we have that those for $g$ are also. We have established that $g$ satisfies all conditions in the Full Irreducibility Criterion, hence represents a fully irreducible outer automorphism.

Further, since $g$ has no periodic Nielsen paths, its ideal Whitehead graph is now the union of the stable Whitehead graphs of $F^p$, hence of $F$, which we have already established is the ideal Whitehead graph for a principal fully irreducible outer automorphism.
\end{proof}

\begin{prop}\label{p:PrincipalGivesLoop}
	Suppose $\vphi$ is a principal fully irreducible outer automorphism in $Out(F_3)$ and $g$ is a train track representative of $\vphi$.
Then some power of $g$'s Stallings fold decomposition is fold-conjugate to one determining a directed loop in the Lonely Direction Automaton $\mA_3$.
\end{prop}

\begin{proof}
Suppose $g\colon \G \to \G$ is a train track representative of a principal fully irreducible outer automorphism $\vphi$. Since $\vphi$ is a principal fully irreducible outer automorphism, it and its powers are lone axis outer automorphisms. By \cite[Corollary 3.8]{loneaxes}, there exists another representative $g'\colon \G' \to \G'$ of $\vphi$ and a power $h=g'^p$ of $g'$ so that $h$ fixes all vertices of $\G'$ and all but one direction of $\G'$. Further, $h$ has precisely one illegal turn and each fold in its Stallings fold decomposition is a different-length fold (or $h$ would identify vertices, hence could not fix all vertices). Since $\vphi^p$ is a lone axis outer automorphism and $h$ represents $\vphi^p$, we know $h$ has no PNPs, implying that $IW(\vphi)$ will be the disjoint union of its stable Whitehead graphs. The stable Whitehead graphs are the restrictions of the local Whitehead graphs to the periodic direction vertices. This tells us:
\begin{itemize}
  \item[(1)] $\G'$ is $5$-edged, with all but a single vertex being of valence-3, and the final vertex (which we call $w$) being of valence $4$, and
  \item[(2)] $LW(h,v)$ is a triangle for each valence-3 vertex $v$, and
  \item[(3)] $LW(h,w)$ contains a triangle, between the vertices representing the 3 fixed directions at $w$. Call the vertex representing the 4th direction $\alpha$.
\end{itemize}

Now, notice that the final different-lengthed fold yields a direction coming from the interior of an edge, so that it can only ever be involved in a single taken turn (see Lemma 2). By the process of elimination, this tells us that $\alpha$, which necessarily represents this direction, is attached to the rest of $LW(h,w)$ by precisely a single edge.

One can construct, as follows, a comprehensive loop in $\G'$ whose image will not be folded by any fold in the Stallings fold decomposition of $h$, even if the Stallings fold decomposition is performed multiple times. Consider any edge $e$ of $\G'$. Because $h$ is expanding Perron-Frobenius, for some $k$, we'll have that $h^k(e)$ will contain at least 2 copies of some edge $e'$ of $\G'$. Consider any turn $\{\overline{e_1}, e_2\}$ taken by $h(e')$. $h^{k+1}(e)$ will contain at least 2 copies of $e_1e_2$. Let $\sigma$ denote the subpath of $h^{k+1}(e)$ starting directly after $e_1$ in a first copy of $e_1e_2$ and then ending directly after $e_1$ in a second copy of $e_1e_2$. Note that $\sigma$ is a loop. Also, since $h$ is a train track map and $\sigma$ is in the $h^{k+1}$-image of an edge, the Stallings folds cannot fold any turns in an image of $\sigma$, so will always be permissible. If $\alpha$ is not comprehensive, since the $h$-image of each edge contains each other edge,  we can replace it by a $h$-image that is comprehensive, we can also do so to ensure that the turns taken by $\sigma$ are precisely those appearing in the $LW(h,v)$, as $v$ varies over the vertices in $V(\G')$. Then $\G'$ satisfies the Lonely Direction Property and the Stallings fold sequence for $h$ gives a loop in $\mA(\G')$ with $\sigma$ as the comprehensive loop.

Since $g$ and $h$ are train track representatives (of powers) of the same lone axis fully irreducible, they share an axis, hence (up to taking a power) have fold-conjugate Stallings fold decompositions and the proposition follows.

\end{proof}

\begin{mainthmA}\label{t:MainTheoremA}
The graphs carrying train track representatives of principal fully irreducible outer automorphisms in $Out(F_3)$ are precisely:
~\\
\vspace{-5mm}
\begin{figure}[H]
\centering
\includegraphics[width=4in]{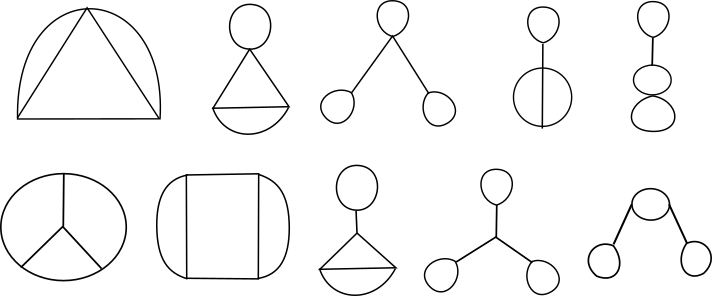}
\label{fig:21Graphs}
\end{figure}
That is, each of the graphs of $CV_3$ where either all vertices are valence-3 or all vertices are valence-3 except that one is valence-4 carries a train track representative a of principal fully irreducible outer automorphism in $Out(F_3)$.
\end{mainthmA}

\begin{proof}
First notice that these are precisely the graphs in the primary maximal strongly connected component of $\mA_3$
(which one may recall we denote by $P\mA_3$),
together with those obtained by partially completing the folds represented by the edges of the component
(we will call these the \emph{edge-determined partial fold graphs}.
The graphs and edges in the other maximal strongly connected components appear also in this component.

Next, notice that, given a directed loop $\mL$ in the automaton,
the Stallings fold sequences obtained by starting the loop at different nodes in $\mL$ yield $\out$-conjugate outer automorphisms.
Since the property of being a principal fully irreducible outer automorphism is a conjugacy class invariant,
the property of an outer automorphism obtained from $\mL$ being a principal fully irreducible outer automorphism
is invariant with respect to the node in $\mL$ one starts traversing the loop from.
Hence, to realize a graph $G$ appearing in the first row of the theorem statement,
we can take the loop $L'$ of Proposition \ref{p:AllLoopsPrincipal},
append to it a loop containing the desired graph $G$ in our automaton
(existing because $L'$ and $G$ are in the same maximal strongly connected component)
and, again using Proposition \ref{p:AllLoopsPrincipal},
obtain a train track representative on $G$ of a principal fully irreducible outer automorphism.

We now prove that the edge-determined partial fold graphs (depicted in the theorem statement's 2nd row)
carry train track representatives of principal fully irreducible elements of $Out(F_3)$.
For an arbitrary such graph $\Gamma'$ represented by a directed edge $\mE$ of $P\mA_3$,
suppose $\Gamma'$ is obtained by partially folding the edge $e'$ over the edge $e$ in a graph $\Gamma$.
Further, suppose $g\colon\Gamma\to\Gamma$ is an irreducible transparent train track representative
of a principal fully irreducible outer automorphism whose Stallings fold decomposition
is given by a directed loop $\mL$ in the automaton starting with a fold of $e'$ over $e$.
Such a loop exists by appending to the loop $L'$ of Proposition \ref{p:AllLoopsPrincipal}
a loop traversing $\mE$, then following the arguments in the previous paragraph.
We claim, as follows, that there is a train track representative $g'\colon\Gamma'\to\Gamma'$
of a principal fully irreducible element of $Out(F_3)$ obtained by conjugating $g$ by this partial fold.

Since $g$ is irreducible and transparent, $length(g(e))\geq |E(\Gamma)|>2$.
Supposing first that the direction of $e$ is fixed, we can write $g(e)=ee_1\cdots e_k$ for some $k\geq 2$.
Subdivide $e$ into subedges so that $e=a_0\cdots a_k$, and $g(a_0)=e$, and $g(a_j)=e_j$ for $j=1, \dots, k$.
Similarly subdivide $e'$ as $a_0'\cdots a_{k+1}'$ with $g(a_0')=e$ and $g(a_j')=e_j$ for $j=1, \dots, k$.
Fold $a_0$ and $a_0'$ to obtain an edge $E$. Call the fold $f$ and let $g'\colon\Gamma'\to\Gamma'$
be the quotient of $g$ under $f$.
Let $g_1$ be such that $g=g_1 \circ f$ (and so that then also $g'= f \circ g_1 $).
If $e'$ is instead fixed, replace $e$ with $e'$ in $g(e)$ and $g(e')$ and then define $g'$ the same.
We show that $g'$ is a train track map.

By construction, $g'$ takes vertices to vertices. We now show each $g'^k$ is locally injective on edge interiors.
Notice that
$g^{k+1}=g_1 \circ g'^k \circ f$.
Suppose for the sake of contradiction that for some edge $E'$ in $\Gamma'$, we have $g'^k(E')$ is not tight.
We show that in this circumstance $g$ would not be a train track map.
We consider separately the cases where $E'=E$ and $E'\neq E$. In the latter case, let $a\in E(\Gamma)$ be such that $E'=f(a)$.
Since $g$ is a train track map,  $g^{k+1}(a)=g_1 \circ g'^k \circ f(a)=g_1 \circ g'^k (E')$ must be tight.
But $g'^k(E')$ was not tight and one cannot map a nontight path to a tight path (in this case via $g_1$),
implying that $g^{k+1}(a)$ could not be tight, a contradiction.
Now suppose $E'=E$. Then $f(e)=E a_1\cdots a_k$ and
$g^{k+1}(e)=g_1 \circ g'^k \circ f(e)=g_1 \circ g'^k (E a_1\cdots a_k)$ must be tight.
But $g'^k(E'=E)$ was not tight, and so $g'^k(E a_1\cdots a_k)$ cannot be tight, and thus $g^{k+1}(e)$ cannot be tight,
again contradicting that $g$ is a train track map.

Since the outer automorphisms represented by $g$ and $g'$ are $Out(F_3)$-conjugate,
$g'$ must also represent a principal fully irreducible outer automorphism.

Proposition \ref{p:PrincipalGivesLoop} proves the reverse containment.

\end{proof}

\section{Outer Space Interpretation}

Culler--Vogtmann Outer space was first defined in \cite{cv86}. We refer the reader to \cite{FrancavigliaMartino,b15,v15} for background on Outer space and give only abbreviated discussion here. For $r\ge 2$ we denote the (volume-one normalized) Outer space for $F_r$ by $CV_r$.  Points of $CV_r$ are equivalence classes of volume-1 marked metric graphs $h:R_r\to\G$ where $R_r$ is the $r$-rose, where $\G$ is a finite volume-1 metric graph with betti number $b_1(\Gamma)=r$ and with all vertices of degree $\ge 3$, and where $h$ is a homotopy equivalence called a \emph{marking}. There is an asymmetric metric $d_{\cv}$ on $CV_r$.

Given a Stallings fold decomposition of a train track map $g$, one can define a ``periodic fold line'' in $CV_r$.
In \cite{stablestrata} it is proved that such periodic fold lines of train track maps are geodesics in the sense that,
given 3 points $\gamma(t_1),\gamma(t_2),\gamma(t_3)$ on the geodesic $\gamma$ with $t_1<t_2<t_3$,
we have that $d_{\cv}(\gamma(t_1),\gamma(t_2))+d_{\cv}(\gamma(t_2),\gamma(t_3))\leq d_{\cv}(\gamma(t_1),\gamma(t_3))$.
One may call such a geodesic $\gamma_g$ an \emph{axis} for $g$ and for $\vphi$, where $\vphi$ is the outer automorphism represented by $g$. This is all explained carefully in \cite{stablestrata}.

$CV_r$ has a simplicial complex structure with some faces missing. It has an open simplex for each marked graph (obtained by varying the lengths on the edges of that graph). The faces of a simplex are obtained by collapsing the edges of a forest (so that their lengths become zero).

The following is an easy Corollary of Main Theorem A.

\begin{mainthmB}
The open simplices in $CV_3$ with principal axes passing through them are precisely those whose underlying graph is listed in Main Theorem A, i.e all those in the top 2 dimensions.
\end{mainthmB}

A few words on the proof: By the proofs of Proposition \ref{p:PrincipalGivesLoop} and Main Theorem A,
we have that the Stallings fold decomposition of any principal fully irreducible element of $Out(F_3)$
will consist only of folds between graphs represented by nodes in $P\mA_3$,
together with possibly part of a fold represented by an edge in $P\mA_3$.
At the start and end of each of the full folds, the axis will pass through the simplex of a 1st row graph.
During the fold (or at the start of end of a partial fold),
the axis will pass through the open simplices of 2nd row graphs.
Varying the marking on a simplex simply $Out(F_3)$-conjugates the outer automorphism,
so does not change whether it is principle fully irreducible outer automorphism.\\

\bibliographystyle{alpha}

\bibliography{PaperRefs}

\end{document}